\documentclass[a4paper,12pt]{amsart}
\usepackage{amsfonts}
\usepackage{mathrsfs}
\usepackage{amsmath,amssymb,latexsym,amsfonts,amscd}

\newcommand{\bfx}{{\bf x}}
\newcommand{\bfy}{{\bf y}}
\newcommand{\bfz}{{\bf z}}
\newcommand{\bfu}{{\bf u}}

\newcommand\Sing{{\text{\rm Sing}}}
\newcommand\ind{{\text{\rm ind}}}
\newcommand\wBl{{{\rm wBl}}}

\newcommand\cE{{\mathcal{E}}}

\newcommand\cU{{\mathcal{U}}}

\newcommand\cX{{\mathcal{X}}}

\newcommand\bC{{\mathbb C}}

\newcommand\bP{{\mathbb P}}

\newcommand\bZ{{\mathbb Z}}

\newtheorem{thm}{Theorem}[section]
\newtheorem{lem}[thm]{Lemma}
\newtheorem{cor}[thm]{Corollary}
\newtheorem{prop}[thm]{Proposition}

\theoremstyle{definition}
\newtheorem{defn}[thm]{Definition}
\newtheorem{setup}[thm]{}

\newtheorem{rem}[thm]{Remark}
\theoremstyle{remark}

\title{Explicit resolution of three dimensional terminal singularities}
\author{Jungkai Alfred Chen}
\dedicatory{Dedicated to Prof. Shigefumi Mori on  his 60th
birthday}

\address{\rm Department of Mathematics, National Taiwan University, Taipei,
106, Taiwan} \email{jkchen@math.ntu.edu.tw}

\address{\rm Research Institute for Mathematical Sciences, Kyoto University, Kyoto 606-8502, Japan}

\begin{document}

\maketitle \pagestyle{myheadings} \markboth{\hfill Jungkai A. Chen
\hfill}{\hfill Explicit resolution of $3$-fold terminal
singularities \hfill}

\begin{abstract} We prove that any three dimensional terminal
singularity $P \in X$ can be resolved by a sequence of divisorial
contractions with minimal discrepancies which are weighted blowups
over points.
\end{abstract}

\section{Introduction}

Terminal singularities are the smallest category that minimal
model program could work in higher dimension. In fact, the
development of minimal model program in dimension three was built
on the understanding of three dimensional terminal singularities:  Reid set up some fundamental results
on canonical and terminal singularities (cf. \cite{C3f, Reid83, YPG}), Mori
classified three dimensional terminal singularities explicitly
(cf. \cite{Mo85}) and then Koll\'ar and Mori proved the existence
of flips by classifying "extremal neighborhood" (cf.
\cite{Mo88,KM92}), which is essentially the classification of
singularities on a rational curve representing extremal ray.
Together with the termination of flips of Shokurov (cf. \cite{Sho85}), one has the minimal model
program in dimension three.

It is interesting, and perhaps of fundamental importance, to know
those birational maps explicitly in minimal model program. For
example, if $X$ is a non-singular threefold and $X \to W$ is a
divisorial contraction to a point then $W$ could have simple
singularities like $(x^2+y^2+z^2+u^2=0)$, $(x^2+y^2+z^2+u^3=0)$, or
a quotient singularity $\frac{1}{2}(1,1,1)$ (cf. \cite{Mo82}). It is
expected that the singularities get worse by further contractions.

On the other way around, given a germ of three-dimensional
terminal singularity $P \in X$, it is expected that one can have a
resolution by successive divisorial contractions. For example,
given a terminal quotient singularity $P \in X$, one has the
"economical resolution" by Kawamata blowups successively. In
\cite{HaII}, Hayakawa shows the following

\begin{thm}\label{Haya} For a
terminal singularity $P \in X$ of index $r>1$, there exists a
partial resolution
$$ X_n \to \ldots \to X_1 \to X_0= X \ni P$$ such that $X_n$ is Gorenstein and
each $f_i: X_{i+1} \to X_{i}$ is a divisorial contraction to a point
$P_i \in X_i$ of index $r_i >1$ with minimal discrepancy $1/r_i$.
All these maps $f_i$ are weighted blowups.
\end{thm}

It is natural to ask whether one can resolve  terminal
singularities  of index $1$ in a similar manner, after Markushevich's result that there exists a divisorial contraction with discrepancy $1$ over any $cDV$ point which is a weighted blowup (cf. \cite{Ma}).
\begin{defn}
Given a three-dimensional terminal singularity $P \in X$. We say
that there exists a feasible resolution for $P \in X$ if there is a
sequence $$ X_n \to X_{n-1} \to \ldots \to X_1 \to X_0=X \ni P,$$
such that $X_n$ is non-singular and each $X_{i+1} \to X_{i}$ is a
divisorial contraction to a point with minimal discrepancy, i.e. a
contraction to a point $P_{i} \in X_i $ of index $r_i \ge 1$ with
discrepancy $1/r_i$.
\end{defn}

The purpose of this note is prove that a feasible resolution
exists for three dimensional terminal singularities.

\begin{thm}[Main Theorem]
Given a three-dimensional terminal singularity $P \in X$. There
exists a feasible resolution for $P \in X$.
\end{thm}

One might expect to construct such resolution by finding a
divisorial contraction with discrepancy $1$ over a terminal
singularity of index $1$ and combining with Theorem \ref{Haya}. However,
there are some technical difficulties.

First of all, given a divisorial contraction $Y \to X \ni P$ with
discrepancy $1$ over a terminal singularity $P\in X$ of index $1$
, then $Y$ usually have singularities of higher indexes. Resolving
these higher index points by Hayakawa's result, one might pickup
some extra  singularities of index $1$ in the process. However,
the studies of singularities of index $1$ was not there in
Hayakaya's work.

Another difficulty is that singular Riemann-Roch is not sensitive
to Gorenstein points. Therefore, the powerful technique introduced
by Kawakita (cf. \cite{Kk01, Kk05, Kk11}) to study singularities and
divisorial contraction by using singular Riemann-Roch formula is
not valid.

What we did in this note is basically pick convenient weighted
blowups, keep good track of terminal singularities, and  put them
into a right hierarchy. The hierarchy is as following: 1. terminal quotient
singularities; 2. $cA$ points; 3. $cA/r$ points; 4. $cD$
and $cAx/2$ poitns; 5. $cAx/4$, $cD/2$, and $cD/3$ points ; 6.
$cE_6$ points; 7. $cE/2$ points; 8. $cE_7$ points; 9. $cE_8$ points.

These involves careful elaborative case-by-case studies. The
reader might find that the structure is very similar to part of
Koll\'ar work in \cite{Ko99}. Indeed, a lots of materials can be found in the preprints of Hayakawa \cite{HaD, HaE}, in which he tried to classified all divisorial contractions with discrepancy $1$ over a $cD$ or $cE$ point. Many of our choices of resolutions are inspired by his works. This work can not be done without his work in \cite{HaD, HaE}. For reader's convenience and for the sake of self-contained, we choose to reproduce the proofs that we needed here.
The existence of divisorial contractions and explicit descriptions are already given by Hayakawa in his series of works. What is really new in this article is that we choose those convenient weighted blowups and work out the inductive process.

{\bf Acknowledgement.} This work was done during the author's
visit to RIMS as Visiting Professor. The author would like thank
the hospitality of RIMS. We would like to thank Prof.  Shigefumi
Mori for the invitation and encouraging discussion during the
preparation of the work. We are indebted to Hayakawa, Kawakita,
Kawamata, Koll\'ar, Matsuki for helpful discussions. The author
would also like to thank Koll\'ar for showing his work \cite{Ko99}.

\section{Preliminaries}

\subsection{weighted blowups}
We will need weighted blowup which are divisorial contraction with
minimal discrepancy. For this purpose, we first fix some notations.

Let $N=\bZ^n$ and $v_0=\frac{1}{r}(a_1,...,a_n)  \in \frac{1}{r}
\bZ^n$. We write $\overline{N}:=N + \bZ v_0$. Let $\sigma$ be the
cone of first quadrant, i.e. the cone generated by the standard
basis $e_1,...,e_n$ and $\Sigma$ be the fan consists of $\sigma$
and all the subcones of $\sigma$. We have that
$\cX_0:=\textrm{Spec} \bC[\sigma^\vee \cap \overline{M}]$ is a
quotient variety of $\bC^n$ by the cyclic group $\bZ/r\bZ$, which
we denote it as $\bC^n/v_0$ or $\bC^n/\frac{1}{r}(a_1,...,a_n)$.

For any primitive vector $v=\frac{1}{r}(b_1,b_2,...,b_n) \in
\overline N$ with $b_i > 0$, we can consider the weighted blowup
$\cX_1 \to \cX_0:= \bC^n/v_0$ with weight $v$, which is the toric
variety obtained by subdivision along $v$. More concretely, let
$\sigma_i$ be the cone generated by
$\{e_1,...,e_{i-1},v_1,e_{i+1},...,e_n\}$, then
$$\cX_1:=\cup_{i=1}^n \cU_i,$$ where $\cU_i =\textrm{Spec} \bC[
\sigma_i^\vee \cap \overline{M}]$. We always denote the origin of
$\cU_i$ by $Q_i$ and the exceptional divisor  $ \cE \cong
\bP((b_1,b_2,...,b_n))$ by $\bP(v)$.

For any semi-invariant $\varphi= \sum \alpha_{i_1,...,i_n}
x_1^{i_1}...x_n^{i_n}$, and  for any vector
$v=\frac{1}{r}(b_1,b_2,...,b_n) \in \overline{N}$ we define
$$wt_v(\varphi):=\min \{ \sum_{j=1}^n \frac{ b_j i_j}{r} | \alpha_{i_1,...,i_n} \ne 0 \}.$$

Let  $X \in \cX_0$ be a complete intersection defined by
semi-invariants $(\varphi_1=...=\varphi_c=0)$. Let $Y$ be its proper
transform in $\cX_1$. By abuse the notation, we also call the
induced map $f: Y \to X$ the weighted blowups of $X$ with weight
$v$, or denote it as $\wBl_v: Y \to X$. Notice that $Y \cap U_i$ is defined by
$\tilde{\varphi}_1=...=\tilde{\varphi}_c=0$ with
$$\tilde{\varphi}_j:=\varphi(x_1x_i^{\frac{a_1}{r_0}}, \ldots,x_{i-1}x_i^{\frac{a_{i-1}}{r_0}},x_i^{\frac{a_i}{r_0}},x_{i+1}x_i^{\frac{a_{i+1}}{r_0}},\dots,x_nx_i^{\frac{a_n}{r_0}})x_i^{-wt_{v_0}(\varphi)},$$
for each $j$. Let $E:=\cE \cap Y \subset
\bP(v)$ denote the exceptional divisor and $U_i:=\cU_i \cap Y$.


Let $X=( \varphi_1=\varphi_2=...=\varphi_c=0) \subset \bC^n/v_0$
be an irreducible variety such that $o=P \in X$ is the only
singularities. Let $Y \to X \ni P$ be a weighted blowup with
weight $v$ and exceptional divisor $E=\cE \cap Y \subset \bP(v)$.
We are interesting in $\text{Sing}(Y)$. We may decompose it into $$
\text{Sing}(Y)= \text{Sing}(Y)_{\ind=1} \cup \text{Sing}(Y)_{\ind>1},$$ where
$\text{Sing}(Y)_{\ind=1}$ (resp. $\text{Sing}(Y)_{\ind>1}$) denotes the locus of
singularities of index $=1$ (resp. $>1$).  Clearly, the locus of
points of index $>1$ in  $\cX_1$ coincide with $\text{Sing}(\bP(v))$.
Hence we have
$$\text{Sing}(Y)_{\ind>1} = Y \cap \text{Sing}(\bP(v))= E \cap \text{Sing}(\bP(v)).$$

We will need the following Lemma to determine singularities on $Y$
of index $1$.

 \begin{lem} \label{key} Keep the notation as above. Consider   $\wBl_v:Y \to X$.
 \begin{enumerate}
 \item  If $\cU_i \cong \bC^n$, then $\text{\Sing}(Y) \cap U_i \subset {\Sing}(E) \cap U_i$.

 \item If $Y$ is a terminal threefold, then ${\Sing}(Y)_{\ind=1} \subset \Sing(E)$
 \end{enumerate}
 \end{lem}

 \begin{proof} For each $i$, we may write  $\varphi_i=
 \varphi_{i,h}+\varphi_{i,r}$, where $\varphi_{i,h}$
 (resp. $\varphi_{i,r}$) denotes the homogeneous part (remaining part) of
 $\varphi_i$ with weight equals to $wt_v(\varphi_i)$.

 In order to prove $(2)$, it suffices to check that $\Sing(Y)_{\ind=1} \cap U_j \subset \Sing(E) \cap U_j$ for
 each $j$. Without loss of generality, we work on $U_n$. 
 For simplicity on notations, we also assume that $X$ is a hypersurface.

 Let $\rho_n: \bC^n \to U_n \cong \bC^n/\mu_r$ be the canonical projection. On $ \bC^n$, $\rho_n^{-1}(Y)$ is defined by semi-invariant
 $$\tilde{\varphi}=\widetilde{\varphi}_{h}+\widetilde{\varphi}_{r} =\varphi_{h}(x_1,...,x_{n-1},1) +\widetilde{\varphi}_{r},$$
 with $x_n |\widetilde{\varphi}_{r}$ and $\rho_i^{-1}(E) \subset \rho_n^{-1}(U_n)$ is defined by $\widetilde{\varphi}_{h}=x_n=0$.

 Note that $\Sing(Y) \subset E$ for $f: Y \to X$ is
 isomorphic away from $p \in X$. Also note that a quotient of a
 three dimensional smooth point can not be terminal
 singularity of index $1$.
 Therefore,
 $$\begin{array}{ll}
\rho_n^{-1}( \Sing(Y)_{\ind=1} \cap U_n) & =\rho_n^{-1}(
\Sing(Y)_{\ind=1} \cap U_n \cap
E) \\
& =\rho_n^{-1}( \Sing(Y \cap U_n)_{\ind=1} ) \cap \rho_n^{-1}(
E) \\
& \subset \Sing( \rho_n^{-1}(Y \cap
U_n)) \cap \rho_n^{-1}(E) \\
 & =\tilde{\varphi}=\tilde{\varphi}_{x_1}...=\tilde{\varphi}_{x_{n}}=x_n=0
 \\
 & \subset
 \tilde{\varphi}=\tilde{\varphi}_{x_1}=...=\tilde{\varphi}_{x_{n-1}}=x_n=0\\
 &=
 \tilde{\varphi}_{h}=\tilde{\varphi}_{h,x_1}=...=\tilde{\varphi}_{h,x_{n-1}}=0
 \\
 &= \Sing(\rho_n^{-1}(E\cap U_n)),
 \end{array}$$

where $\tilde{\varphi}_{x_i}$ denotes $\frac{\partial
\tilde{\varphi} }{\partial x_i}$.

 Since $\rho_n$ is \'etale, therefore, $\rho_n (\Sing(\rho_n^{-1}(E\cap
 U_n))) \subset \Sing(E \cap U_n)$. The statement follows for
 hypersurface.
The same argument works for higher codimension as well.

The proof for $(1)$ also follows from the similar argument.
 \end{proof}

\begin{cor} \label{qsmooth} Let $Y$ be  a terminal threefold obtained by  $\wBl_v: Y \to X$ with weight
$v$. Suppose that $E$ is a quasi-smooth weighted complete
intersection in $\bP(v)$, then $\Sing(Y)_{\ind=1} = \emptyset$.
\end{cor}

\begin{proof}
If $E$ is a quasi-smooth weighted complete intersection in
$\bP(v)$, then $\Sing(E)= E \cap \Sing(\bP(v))$. Therefore,
$$\Sing(Y)_{\ind=1} \subset \Sing(E) \subset \Sing(\bP(v)).$$
However, $\Sing(\bP(v))$ consists of quotient singularities of
index $>1$.
 This implies
in particular that $\Sing(Y)_{\ind=1} =\emptyset$.
\end{proof}

 \subsection{weighted blowup of threefolds}
 Given a threefold terminal singularity $P \in X=(\varphi=0) \subset \bC^4/v_0$ of index $r$, we usually consider
 weighted blowup $\wBl_v:Y \to X$ with weight $v=\frac{1}{r}(a_1,a_2,a_3, a_4)$ and $a_i
 \in \bZ_{>0}$ \footnote{ Divisorial contractions to a  point of index $r>1$ have been studied extensively by Hayakawa
 and are known to be weighted blowups}. It worths to determine when
 a weighted blowup is a divisorial contraction.

\begin{thm} \label{terminal}
 Let $P \in X=(\varphi=0) \subset \bC^4$ be a germ of three dimensional terminal singularity and
 $f=\wBl_v :Y \to X$  with weight
 $v=\frac{1}{r}(a_1,a_2,a_3,1)$ with exceptional divisor $E \subset \bP(v)$. Suppose that

\begin{itemize}
\item $E$ is irreducible;

\item  $\frac{1}{r}\sum a_i  -wt_v(\varphi)-1=\frac{1}{r}$;

\item  either $Y \cap U_4$ is terminal or  $E$ has  Du Val singularities on $U_4$.
 \end{itemize}  Then $Y \to X$ is a divisorial
 contraction.

 Moreover, $\Sing(Y)_{\ind=1} \cap U_4 \subset \Sing(E) \cap U_4$. For any
 $R \in \Sing(Y)_{ind=1} \cap U_4$, $R$ is at worst of type $cA$ (resp. $cD,
 cE_6, cE_7,cE_8$) if $R \in \Sing(E)$ is of type $A$ (resp. $D,
 E_6,E_7,E_8$).
\end{thm}

\begin{proof}
Suppose that $E$ is irreducible, then $K_Y=f^*K_X+a(E,X)E$ with
$a(E,X)=\frac{1}{r}\sum a_i
 -wt_v(\varphi)-1$. Let $D=(x_4=0) \subset Div(X)$ and $D_Y$ be its
 proper transform in $Y$. One has $f^*D=D_Y+\frac{1}{r} E$.
Hence $\frac{1}{r}=\frac{1}{r}\sum a_i -1
 -wt_v(\varphi)$ implies that $$f^*(K_X+D)=K_Y+D_Y$$ and hence $D_Y
 \sim_X -K_Y$.

Let $g:Z \to Y$ be a resolution of $Y$. For any exceptional divisor
$F$ in $Z$ such that $g(Z) \subset D_Y$. One has $g^*D_Y=D_Z+mF+\ldots$
for some $m>0$. It follows that $a(F,Y)=m >0$. Hence $Y$ is terminal
if $Y -D_Y= Y \cap U_4$ is terminal.

In fact, by  Lemma \ref{key}.(1), one sees that
 $\Sing(Y) \cap U_4 \subset \Sing(E) \cap U_4$. If $E \cap U_4$ is DuVal, then $\Sing(E) \cap U_4$
is isolated hence so is $\Sing(Y) \cap U_4$. More precisely, for $R
\in \Sing(E) \cap U_4$ with local equation
$$\psi:=\varphi_h(x_1,...,x_4)|_{x_4=1}$$ of type $A$ (resp. $D,E$),
the local equation for $R \in Y \cap U_4$ is of the form
$$ \psi + x_4 g(x_1,...,x_4),$$
which is a compound DuVal equation. Therefore, if $R$ is singular in
$Y$, then $R$ is a at worst isolated cDV of type $cA$ (resp. $cD,
cE$) \footnote{If there are lower degree terms appearing in $g$,
then the singularity $R$ could be simpler or even non-singular}. By
results of Reid \cite{C3f},  Koll\'ar and Shephard-Barron \cite{KSB}, an isolated cDV singularity is
terminal. Hence $Y$ is terminal and therefore $f:Y \to X$ is a
divisorial contraction.
\end{proof}

In the sequel, all weighted blowups with discrepancy $1/r$ over a
terminal singularity  of index $r$ can easily checked to be
divisorial contractions with minimal discrepancy $1/r$ by the above
Theorem \ref{terminal} or by direct computation.



We will need some further easy but handy Lemmas.

\begin{lem} \label{u1}
Let $\wBl_v: Y \to X$ be a divisorial contraction. If
$wt_v(x^2)=wt_v(\varphi)$, then $\Sing(Y) \cap U_1 = \emptyset$.
\end{lem}

\begin{proof}Now  $E$ is defined by $(\Phi: {\bf x}^2 +
f({\bf y} , {\bf z}, {\bf u})=0 ) \subset \bP(v)$.
It is clear that $\Sing(Y)_{\ind=1} \cap U_1 \subset \Sing(E) \cap U_1= \emptyset$.
Notice also that $Q_1$ is the only point in $U_1$ with index $>1$ and $Q_1 \not \in Y$. Hence $\Sing(Y) \cap U_1 = \emptyset$.
\end{proof}

\begin{lem} \label{nonsing} Let $\wBl_v: Y \to X$ be a divisorial contraction. If  ${ x}_i^m { x}_j \in {\varphi}$ with $wt_v({
x}_i^m { x}_j)= wt_v(\varphi)$  or ${ x}_i^m  \in {\varphi}$ with
$wt_v({ x}_i^m )= wt_v(\varphi)+1$, then $Y \cap U_i$ is
non-singular away from $Q_i$ and  $Q_i$ is either non-singular or
a terminal quotient singularity of index $wt_{v}(x_i)$. In particular, $\Sing(Y)_{\ind=1} \cap U_i =
\emptyset$.
\end{lem}

\begin{proof}
On $U_i$, $Y\cap U_i$ is given by $(\tilde{\varphi}=0) \subset \bC^4/\bZ_{a_i}$ with
$$ \tilde{\varphi}= \left\{ \begin{array}{l}  x_j+\text{others}, \text{ if }  wt_v({x}_i^m { x}_j)= wt_v(\varphi); \\
x_i+\text{others}, \text{ if }  wt_v({ x}_i^m )= wt_v(\varphi)+1. \end{array} \right.$$
Hence $Y \cap U_i \cong \bC^3/\bZ_{a_i}$ and the statement follws.
\end{proof}


\begin{lem} \label{cA}
Consider  $X= (\varphi=0) \subset \bC^4$. Suppose that $R \in X$
is an isolated singularity and  $xy \in \varphi$ or $x^2+y^2 \in
\varphi$. Then $R$ is of cA-type.
\end{lem}

\begin{proof}
Up to a unit, we may assume that $\varphi=xy+xg(x,z,u)+yh(y,z,u)+\bar{\varphi}(z.u)$.
Since $\frac{\partial^2 \varphi}{\partial x \partial y}(R)=1 \ne 0$,
the local expansion near $R$ is of the form $\bar{x}
\bar{y}+\bar{f}(\bar{z},\bar{u})$ where $\bar{x}=x-x(R)$
respectively. Also $mult_{o} \bar{f} \ge 2$. Hence it is a $cA$
point.
\end{proof}

\begin{cor} \label{cAq}
Consider $X=(\varphi=0) \subset \bC^4/ \bZ_r$. Suppose that $R \in
\Sing(X)_{\ind=1}$ is an isolated singularity and either $xy  \in
\varphi$ or $x^2+y^2 \in \varphi$. Then $R$ is of $cA$-type.
\end{cor}

\begin{proof}
Let $\pi : \bC^4 \to \bC^4/\bZ_r$ is the quotient map. Since $R$
is a index $1$ point, $\pi^{-1}(R)$ does not intersects the fixed
locus of the $\bZ_r$ action. This implies in particular that
$\bZ_r$ acts on $\pi^{-1}(R)$ freely and each point of $Q \in
\pi^{-1}(R)$ is singular in $\pi^{-1}(X)$. By Lemma \ref{cA}, $Q$
is of type $cA$ and hence so is $R$.
\end{proof}

By the similar argument, one can also see the following
\begin{lem} \label{cDq}
Consider $X=(\varphi=0) \subset \bC^4/ \bZ_r$ with $r=1,2,3$. Suppose that
$\varphi=x^2+f(y,z,u)$ with $f_3$,  the $3$-jet  of $f$, is nonzero and not a cube. Let $R \in
\Sing(X)$ be an isolated singularity. Then $R$ could only be of type $cA, cA/r, cD, cD/r$ or a terminal quotient singularity.
\end{lem}

\section{resolution of $cA$ and $cA/r$ points}

\begin{lem} \label{quo}
Let $f: Y \to X$ be the economic resolution of   a terminal quotient
singularity $P \in X$. Then this is a feasible resolution
for $P \in X$.
\end{lem}

\begin{proof}
Given a terminal quotient singularity $P \in X$ of type
$\frac{1}{r}(s,r-s,1)$ with $s<r$ and $(s,r)=1$, we start by
considering the Kawamata blowup $Y \to X$ (cf. \cite{Ka}), i.e.
weighted blowup with weight $v=\frac{1}{r}(s,r-s,1)$ over $P$. It
is clearly a divisorial contraction with minimal discrepancy
$\frac{1}{r}$.

Note that $\Sing(Y)$ consists of at most two points $Q_1,Q_2$ of
index $s,r-s$ respectively.  By induction on $r$, we get a
resolution $Y=Y_{r-1} \to \ldots \to Y \to X \ni P$. It is easy to
see that this is the economic resolution.
\end{proof}

\begin{thm} \label{cArsln}
There is a feasible resolution  for any $cA$ points.
\end{thm}
\begin{proof}
For any $cA$ point $p \in X$, there is an embedding $j: X \subset
\bC^4$ such that $P \in X$ is defined by $(\varphi: xy+f(z,u)=0)
\subset \bC^4$. We fix this embedding once and for all and define
$$\tau(P):=\min\{ i+j| z^iu^j \in f(z,u)\}.$$
We may and do assume that $z^\tau \in f$. We write $f=f_\tau+f_{>
\tau}$, where $f_{\tau}$ denote the homogeneous part of weight
$\tau$.

We need to introduce
$$ \tau^\sharp(P):=\min\{ i+j | z^iu^j \in f(z,u), i \le 1\}.$$

Since $P\in X$ is isolated, one has that $f$ contains a term of
the type $zu^{p-1}$ or $u^p$ for some $p$. Hence $\tau^\sharp(P)$
is well-defined. Notice also that $\tau(P) \le \tau^\sharp(P)$.

We shall prove by induction on $\tau$ and $\tau^\sharp$.


\noindent {\bf Case 1.} $\tau=2$.\\ By easy change of coordinates,
we may and do assume that $f(z,u)=z^2+u^b$. We take $Y \to X$ to
be the weighted blowup with weights $(1,1,1,1)$ (or the usual
blowup over $P$). It is clear that $\Sing(Y)=\{Q_4\}$, which is
defined by $$ (\tilde{\varphi}: xy+z^2+u^{b-2}=0)\subset \bC^4.$$
By induction on $b$, we are done.




\noindent
{\bf Case 2.} $ \tau >2$.\\
We may write $f_\tau= \prod (z-\alpha_t u)^{l_t}$ since $z^\tau
\in f$.  We take $\wBl_v: Y_1 \to X$  with
weights $v=(1,\tau-1,1,1)$. It is clear that
$\Sing(Y)_{\ind>1}=\{Q_2\}$, which is a terminal quotient
singularity of index $\tau-1$. Hence it remains to consider index
$1$ points.

We have that $$E =(\bfx \bfy+ \prod (\bfz-\alpha_t \bfu)^{l_t}=0)
\subset \bP(1,\tau-1,1,1).$$

Now $\Sing(E)=\{R_t= (0,0,\alpha_t,1)\}_{l_t \ge 2}$. In fact, for
any $R_t$ with $l_t \ge 2$, one sees that $R_t \subset E$ is a
singularity of $A$-type, it follows that if $R_t$ is singular in
$Y$, then it is of type $cA$ with $\tau(R_t) \le l_t$.

\noindent
{\bf Subcase 2-1.} $f_{\tau}$ factored into more than one factors.\\
Then $\tau(R_t) \le l_t < \tau$, then we are done by induction on
$\tau$.

\noindent
{\bf Subcase 2-2.} If $f_\tau$ factored into only one factor.\\
We may assume $f_\tau=z^\tau$ by changing coordinates. It is easy
to see that $\tau(Q_4) \le \tau^\sharp(Q_4) < \tau^\sharp(P)$
hence $$ \tau(Q_4)+\tau^\sharp(Q_4)< \tau(P)+\tau^\sharp(P).$$
Then we are done by induction on $\tau+\tau^\sharp$.
\end{proof}

\begin{cor} \label{cAr}
There is a feasible resolution for any $cA/r$ point.
\end{cor}

\begin{proof}
Given $P \in X$ defined by
$$ (\varphi: xy+f(z^r,u)=xy+\sum a_{ij}z^{ir}u^j=0 )\subset \bC^4/\frac{1}{r}(s,r-s,1,r).$$

Let $$ \left\{ \begin{array}{l} \kappa^\sharp(\varphi):=\min\{k| u^k \in f\}, \\
  \kappa(\varphi):=\min\{ i +j | z^{ir}u^j \in f\}. \end{array} \right.$$
  We shall prove by
induction on $\kappa^\sharp+\kappa$. Note that there is some $u^k
\in f$ otherwise $P$ is not isolated. Thus $\kappa^\sharp+\kappa $
is  finite and $\kappa \le \kappa^\sharp$.

{\bf 1.} $\kappa^\sharp=1, \kappa=1$.\\
Then $P \in X$ is a terminal quotient  singularity. We are done.

{\bf 2.} $\kappa^\sharp+\kappa>2$.\\
We always consider $Y \to X$ the weighted blowup with weights
$\frac{1}{r}(s,  \kappa r-s, 1,r)$, which is a divisorial
contraction by \cite{HaI}. Computation on each charts similarly,
one sees the following:
\begin{enumerate}
\item $Y \cap U_1$ is singular only at $Q_1$, which is
 a terminal quotient singularity of index $s$ (non-singular on $U_1$ if $s=1$).

\item $Y \cap U_2$ is singular only at $Q_2$, which is a terminal
quotient singularity of index $\kappa r-s$ (non-singular on $U_2$
if $\kappa r-s=1$).

\item $Y \cap U_3$ is defined by $xy+f(z, uz) z^{-\kappa}=0
\subset \bC^4$. Hence $\Sing(Y) \cap U_3$ must be of type $cA$ by
Lemma \ref{cA}. There exists a feasible resolution over these
points.


\item it remains to consider $Q_4$, which is locally defined by
$$ (\tilde{\varphi}:xy+\sum a_{ij}z^{ir}u^{i+j-\kappa}=0) \subset \bC^4/\frac{1}{r}(s,r-s,1,r).$$
\end{enumerate}

In fact,  one sees that
$\kappa^\sharp(Q_4)=\kappa^\sharp(P)-\kappa(P)$ and $\kappa(Q_4) \le \kappa(P)$.

By induction on $\kappa^\sharp+\kappa$, we have a feasible resolution  over $Q_4$.
Together with feasible resolution over other singularities on $Y$, we have a feasible resolution over $Y$ and hence over $X$.
\end{proof}

\section{resolution of $cD$ and $cAx/2$ points}
Given a $cD_n$ point $P \in X$ which is defined by $(\varphi: x^2+y^2z+z^{n-1}+u
g(x,y,z,u)=0) \subset \bC^4$ for some $n \ge 4$.
We start by considering the {\it normal form} of $cD$
singularities.

\begin{defn} We say that a $cD$ point $P \in X$ admits a  normal form if
there is an embedding
$$ (\varphi: x^2+y^2z+\lambda yu^l+f(z,u)=0)\subset \bC^4 $$
with the following properties:
\begin{enumerate}

\item $l \ge 2$. (We adapt the convention that $l=\infty $ if
$\lambda=0$.)

 \item $zu^{p-1} \in f$ or $u^p \in f$ for some $p>0$
if $\lambda = 0$.

\item $z^{q-1}u \in f$ or $z^q \in f$ for some $q>0$.

\end{enumerate}

 An isolated
singularity $P \in X$ given by this form (with $l \ge 0$ and
possibly not of $cD$ type) is called a $cD$-like singularity,
which is terminal.

For a $cD$-like singularity $P \in X$, we define
$$\left\{\begin{array}{l}
\mu^\sharp(P \in X):= \min \{2i+j | z^i u^j \in \varphi, i=0 \text{ or } 1\}; \\
 \mu(P \in X):= \min \{ 2i+j |z^i u^j \in \varphi\}; \\
\mu^\flat(P \in X):= \min\{ \mu(P \in X), 2l-2\}; \\
\tau^\sharp(P \in X):= \min \{i+j | z^i u^j \in \varphi, i=0 \text{ or } 1\}.\\
 \end{array} \right.
$$
\end{defn}

Clearly, one has $\mu^\flat \le \mu \le \mu^\sharp \le \infty$. Also $\mu^\sharp, \tau^\sharp < \infty$ if $\lambda =0$.

\begin{lem} \label{mu3}
Given a $cD$-like point $P \in X$ defined by $$(\varphi \colon
x^2+y^2z+\lambda yu^l+f(z,u)=0) \subset \bC^4,$$ with $\mu^\flat \le 3$.  Then  there exists a feasible
resolution for $P \in X$.
\end{lem}

\begin{proof}
If there is a linear or quadratic term in $f$, then $P$ is
non-singular or of $cA$-type by Lemma \ref{cA}. In particular,
feasible resolution exists.
 We thus assume that $l \ge 2$ and we may write $$f(z,u)=f_3(z,u)+f_{\ge 4}(z,u),$$
where $f_3(z,u)$ (resp. $f_{\ge 4}(z,u)$) is the $3$-jet (resp. $4$ and higher jets) of $f(z,u)$.

\noindent
{\bf Case 1.} $l \ge 3$.\\
If $\mu \le 2$, then $P$ is at worst of
type $cA$. Thus we may and do assume that
 $\mu=3$ and hence $u^3 \in f_3 \ne 0$. Clearly,
$\varphi_3=y^2z+f_3$ is irreducible.

\noindent
{\bf Subcase 1-1.} $f_3$ is factored into more than one  factors.\\
We consider  $\wBl_v: Y \to X$ with weight $v=(2,1,1,1)$.
One can verify that  $E =( \bfy^2 \bfz+ f_3(\bfz,\bfu)=0) \subset \bP(2,1,1,1)$ is
 irreducible. By Lemma \ref{nonsing}, one has that $\Sing(Y)
\cap U_i$ is non-singular away from $Q_i$ for $i=1,2$. In fact,
$Y$ is non-singular at $Q_2$. Therefore, $Y \to X$ is a divisorial
contraction by Theorem \ref{terminal}.

Since $Q_4 \not \in Y$, it
remains to consider $Y \cap U_3$, which is defined by
$$\begin{array}{ll} \tilde{\varphi}: & x^2z+ y^2+\lambda
yz^{l-2}u^{l}+\tilde{f}_3(z,u)+\tilde{f}_{\ge 4}\\
& =x^2z+y^2 +\lambda yz^{l-2}u^{l}+ \prod ( u-
\alpha_t)^{\l_t}+\tilde{f}_{\ge 4},\end{array}$$

where $\tilde{f}_3(z,u)$ (resp. $ \tilde{f}_{\ge 4}(z,u)$ ) denotes the proper transform of $f_3(z,u)$ (resp. $f_{\ge 4}(z,u)$).
More explicitly, $\tilde{f}_3(z,u) = f_3(z,zu)z^{- wt_v(\varphi)}$.

Let $R$ be a singular point in $ \Sing(Y) \cap U_3$. If $f_3$ is factored into more than
one  factors, then $l_t \le 2$ for all $t$. It follows that
$R$ is at worst of $cA$ type  by Lemma \ref{cA}. Notice also that
$\Sing(Y)_{\ind>1}=\{Q_1\}$ which is a quotient singularity of index
$2$. Thus feasible resolution exists.

\noindent
{\bf Subcase 1-2.} $f_3$ is factored into one factor.\\
We thus assume that $f_3=(u-\alpha z)^3$. Change coordinate by
$$\left\{ \begin{array}{l}
\bar{u}=u- \alpha z; \\
\bar{y}=y+\frac{\lambda}{2} \sum_{j \ge 1} C^l_j  \alpha^j z^{j-1}
\bar{u}^{l-j};
\end{array} \right.$$
then we have $$\varphi= x^2+\bar{y}^2 z +\lambda \bar{y} \bar{u}^l+
\bar{u}^3 + \bar{f}_{\ge 4} (z, \bar{u}).$$ Therefore, we may and do
assume that $f_3=u^3$ in the normal form.

We consider again  $\wBl_v: Y \to X$ with weight
$v=(2,1,1,1)$. Since $$E=(\bfy^2 \bfz + \bfu^3=0) \subset
\bP(2,1,1,1),$$ therefore  $Y \to X$ is a divisorial contraction
by Theorem \ref{terminal} (where $U_4$ is replaced by $U_2$).
Moreover,  $\Sing(E) \subset (\bfy=\bfu=0)$, hence $\Sing(E) \subset
U_1 \cup U_3$. Together with Lemma \ref{nonsing}, it remains to
consider $Q_3$, which is a $cD$ point given by
$$ \tilde{\varphi}: x^2z+y^2+\lambda y u^l
z^{l-2}+u^3+\tilde{f}_{\ge 4}.$$

Change coordinate by $\bar{y}:=x, \bar{x}:=y+\frac{1}{2}\lambda u^l=z^{l-2}$, one sees that $Q_3$ is $cD$-like given by
$$ \tilde{\varphi}: \bar{x}^2+\bar{y}^2 z+u^3+\tilde{f}_{\ge 4}-\frac{1}{4}\lambda^2 u^{2l}
z^{2l-4}.$$

Clearly, $Q_3$ is still in Subcase 1-2 and $\tau(Q_3) \le \tau(P)-2$. By induction on $\tau$, we conclude that feasible resolution exists for this case.

\bigskip
\noindent
{\bf Case 2.} $l=2$.\\
{\bf Subcase 2-1.} $f_3 = 0$.\\
We consider  $\wBl_v: Y \to X$ with
weight $v=(2,2,1,1)$. Now $$E=(\bfx^2  + \lambda \bfy \bfu^2+
f_4(\bfz,\bfu)=0) \subset \bP(2,2,1,1)$$ is clearly irreducible.
By considering $Y \cap U_4$, which is nonsingular, one has that $Y
\to X$ is a divisorial contraction by Lemma \ref{nonsing} and
Theorem \ref{terminal}.

One sees that  $\Sing(E) \subset (\bfx=\bfu=0)$, hence $\Sing(E)
\subset U_2 \cup U_3$. Notice also that $Q_1 \not \in Y$ and $Q_2
\in Y$ is a $cA/2$ point. Together with Lemma \ref{cA}, it remains
to consider $Q_3$, which is a $cD$-like point and still in Subcase 2-1. Clearly, $\tau(Q_3) \le \tau(P)-2$. By induction on $\tau$, we conclude that feasible resolution exists for this case.


\noindent
{\bf Subcase 2-2.} $f_3 \ne 0$ and  $\varphi_3=y^2z+\lambda yu^2+f_3$ is irreducible.\\
We consider  $\wBl_v: Y \to X$ with weights $v=(2,1,1,1)$.
Now $$E =( \bfy^2 \bfz+\lambda \bfy \bfu^2+ f_3(\bfz,\bfu)=0)
\subset \bP(2,1,1,1).$$ One has that $Y \to X$ is a divisorial
contraction by the same reason. By Lemma \ref{nonsing}, it's clear
that $\Sing(Y) \cap (U_1 \cup U_2 \cup U_4)=\{Q_1\}$, which is a
quotient singularity of index $2$.

It remains to consider $Q_3$.
 If $f_3$ is factored into more than one factors then the same
argument in Subcase 1-1 works. We thus assume that
$f_3=(\beta u+\alpha z)^3$. In fact, one sees that $Q_3$ is
singular only when $f_3=u^3$. Argue as in Subcase 1-2. We
have a feasible resolution for $P \in X$.

\noindent
{\bf Subcase 2-3.} $f_3 \ne 0$ and  $\varphi_3= y^2z+\lambda yu^2+f_3$ is reducible.\\
In this situation, $y^2z+\lambda yu^2+f_3 = (y+l(z,u))(zy+\lambda
u^2-l(z,u)z)$ for some linear form $l(z,u) \ne 0$. Let
$\bar{y}=y+l(z,u)$, then we have $$\varphi_3=\bar{y}^2z+\lambda
\bar{y}u^2-2\bar{y}zl(z,u).$$ We consider weighted blowup $Y \to
X$ with weights $(2,2,1,1)$. Now $$E =(\bfx^2+ \lambda \bfy
\bfu^2 -2 \bfy \bfz l(\bfz,\bfu)+f_4(\bfz,\bfu)=0) \subset
\bP(2,2,1,1),$$ is clearly irreducible. By considering $Y \cap
U_4$, which is nonsingular by Lemma \ref{nonsing}, one has that $Y \to X$ is a divisorial
contraction  by Theorem \ref{terminal}.
Since $l(z,u) \ne 0$, one sees that $Y \cap U_2$ has at worst
$cA/2$ singularities and $Y \cap U_3$ has at worst $cA$
singularities. Therefore feasible resolution exists.
\end{proof}

By \cite[Proposition 1.3]{Ma}, we have that
\begin{enumerate}
\item if $P \in X$ is $cD_4$, then $ \varphi =
x^2+\varphi_3(y,z,u)+\varphi_{\ge 4}(y,z,u)$ with $\varphi_3(y,z,u)$ is not
divisible by a square of a linear form;

\item  if $P \in X$ is
$cD_n$ with $n \ge 5$, then $ \varphi = x^2+y^2z+\varphi_{\ge
4}(y,z,u)$.
\end{enumerate}

Therefore, the plan is as following: for $cD_4$ points, the parallel argument as in Lemma \ref{mu3} works. For $cD_n \ge 5$ points, which always admits normal forms,
we prove by induction on $\mu^\flat$. We will need to consider  $cAx/2$ points simultaneously in the induction.

\begin{prop}
There is a feasible resolution  for any $cD_4$ singularity.
\end{prop}

\begin{proof} We have $y^2z, z^3 \in \varphi_3$.
Replacing $z$ by $z+u$ and completing square, we may and do assume that
$\varphi_3=y^2z + \lambda yu^2 + f_3(z,u)$,
with $z^3 \in f_3$ and
$$ \varphi = x^2 + y^2z + \lambda yu^2 + f_3(z,u) + \varphi_{\ge 4}(y,z,u).$$

\noindent
{\bf Case 1.} $ \lambda =0$ and $\varphi_3$ is irreducible.\\
We can work as in Subcase 1-1 and 1-2 of Lemma \ref{mu3}.

\noindent
{\bf Case 2.} $\lambda =0$ and $\varphi_3$ is reducible.\\
  In this situation, $y^2z+f_3 = z(y^2+q(z,u))$ for some quadratic form $q(z,u) \ne 0$.
  We consider weighted blowup $Y \to X$  with weights $v=(2,1,2,1)$.
Now $$E =(\bfx^2+  \bfy^2 \bfz+\text{possibly others}=0) \subset
\bP(2,2,1,1),$$ is clearly irreducible. By considering $Y \cap
U_2$, which is nonsingular by Lemma \ref{nonsing}, one has that $Y \to X$ is a divisorial
contraction by Theorem \ref{terminal}.

Since $z^3 \in \varphi_3$, one sees that $Y \cap U_3$
define by $$(\tilde{\varphi}: x^2+y^2+z^2+\text{others}=0) \subset \bC^4/\frac{1}{2}(2,1,1,1),$$
which has at worst $cA/2$ singularities.

It remains to consider $Q_4$. Since $P \in X$ is isolated, one sees that there exists  $yu^p, zu^p$ or $u^p \in \varphi$ for some $p$.
It follows that there exists $yu^{p-1}, zu^{p-2}$ or $ u^{p-3} \in \tilde{\varphi}$ in  $Y \cap U_4$.
Hence feasible resolution exists by induction on $p$.

\noindent
{\bf Case 3.}   $\lambda \ne 0$.
We can work as in Subcase 2-2 and 2-3 of Lemma \ref{mu3}. Note that $z^3 \in f_3$ hence Subcase 2-1 can not happen.
\end{proof}

\begin{lem} A singularity $P \in X$ of type $cD_{n \ge 5}$ admits a normal form with $l \ge 3$.
\end{lem}

\begin{proof}
It is straightforward to solve for formal power series $
\bar{y}=y+y_2+y_3+\ldots$ and $\bar{z}=z+z_2+z_3+\ldots$
satisfying
$$x^2+\bar{y}^2 \bar{z} + \lambda \bar{y} u^l + f(\bar{z},u) =  x^2+y^2z+g_{\ge 4}(y,z,u),$$
where $y_k=y_k(z,u)$ and $z_k=z_k(y,z,u)$ are the $k$-th jets and
$l=\min \{k| yu^k \in g_{\ge 4}(y,z,u)\}$. By Artin's
Approximation Theorem \cite{Ar}, this gives an embedding as
desired.

Observe that $\varphi=0$ is
singular along the line $(x=y=z=0)$ if $z^2|f(z,u)$.   Similarly,
$\varphi=0$ is singular along the line $(x=y=u=0)$ if
$u^2|f(z,u)$.
Since $P \in X$ is isolated, it follows that $z^{q-1}u \in \varphi$ or $z^q \in \varphi$ for some $q>0$ and $zu^{p-1} \in \varphi$ or $u^p \in \varphi$ for some $p>0$ if $\lambda = 0$.
\end{proof}

In order to obtain a feasible resolution for $cD_{n \ge 5}$ points in general, we will need to consider
$cAx/2$ point as well.  Given a $cAx/2$ point $P \in X$, with an
embedding
$$(\varphi \colon x^2+y^2+f(z,u)=0) \subset \bC^4/\frac{1}{2}(1,0,1,1),$$ we
define
$$ \tau(P \in X):= \min \{
i+j | z^iu^j \in f\}.$$ Note that $f(z,u)$ is $\bZ_2$-invariant and hence consists of even degree
terms only. We set $\tau':=\tau/2 \in \bZ$.

For inductive purpose, we start by considering  points with $\tau$ small.

\begin{lem}
Given a $cAx/2$-like point $P$ defined by
$$(\varphi \colon x^2+y^2+f(z,u)=0) \subset
\bC^4/\frac{1}{2}(1,0,1,1),$$ with $\tau \le 2$. Suppose that $P$
is terminal. Then $P$ is non-singular  or $cA/2$. In any case,
feasible resolutions exist for such points.
\end{lem}

\begin{proof}
If $\tau =0$, then $P$ is clearly non-singular. If $\tau=2$, then
we may assume that $z^2 \in f(z,u)$. Hence it is a $cA/2$ point.
By Corollary \ref{cAr}, feasible resolution exists.
\end{proof}

We are now ready to handle  $cAx/2$ and $cD$ points.

\begin{prop} Given a cAx/2 point $P \in X$ with $\tau( P\in
X)=\tau_0 \ge 4$. Suppose that feasible resolutions exist for  $cD$-like point
with $\mu^\flat < \tau_0$ and feasible resolutions exist for $cAx/2$-like
point with $\tau < \tau_0$. Then there is a feasible resolution for $P
\in X$.
\end{prop}

\begin{proof}Let $f_{\tau_0}(z,u)=\sum_{i+j=\tau_0} a_{ij} z^iu^j$. It can be factored
into $ \prod_{t \in T} (\alpha_t z + \beta_t u)^{m_t}$, where $m_t$ denotes
the multiplicities with $\sum_t m_t= \tau_0$.

\noindent
{\bf Case 1.} $f_{\tau_0}(z,u)$ is not a perfect square.\\
 Depending on the parity
of $\tau_0/2$, we first consider  $\wBl_v: Y \to
X$ with weights
$v=\frac{1}{2}(\frac{\tau_0}{2}, \frac{\tau_0}{2}+1,1,1)$ or
$\frac{1}{2}(\frac{\tau_0}{2}+1, \frac{\tau_0}{2},1,1)$. It is
a divisorial contraction with minimal discrepancy $\frac{1}{2}$ (cf.
\cite[Theorem 8.4]{HaI}). Without loss of generality, we study the first case.

Now $E=(\bfx^2+
f_{\tau_0}(\bfz,\bfu)=0) \subset
\bP(\frac{\tau_0}{2}, \frac{\tau_0}{2}+1, 1,1)$.
 Easy computation yields the following:
\begin{enumerate}
\item  $Y \cap U_1$ is non-singular and
$Y \cap U_2$ is singular only at $Q_2$, which is
a terminal quotient  singularity of index $\frac{\tau_0}{2}+1$.

\item $\Sing(Y)_{\ind=1} \cap U_3  \subset \{ R_t\}_{ m_t \ge 2,
(\alpha_t,\beta_t) \ne (1,0)}$. Each $R_t$ is defined by
$x^2+y^2z+unit\cdot \bar{u}^{m_t}+\tilde{f}_{\ge \tau_0}(z,
\bar{u})=0 \subset \bC^4$,  where $\bar{u}:=u+\alpha_t$.  This
is a  $cD$ point with $\mu^\flat(R_t) \le m_t $.

\item Similarly, $\Sing(Y)_{\ind=1} \cap U_4  \subset \{ R_t\}_{
m_t \ge 2, (\alpha_t,\beta_t) \ne (0,1)}$. Each $R_t$ is defined
by $x^2+y^2u+unit\cdot \bar{z}^{m_t}+\tilde{f}_{\ge
\tau_0}(\bar{z},u)=0 \subset \bC^4$, where
$\bar{z}:=z+\beta_t/\alpha_t$.  This is a $cD$ point with
$\mu^\flat(R_t) \le m_t$.

\end{enumerate}

As a summary, one sees that $\Sing(Y) \subset \{Q_2, R_t\}_{m_t
\ge 2}$. Notice  $|T|=1$ would implies that $f_{\tau_0}$ is a
perfect square, which is a contradiction. Hence we may assume that
$|T|>1$ and therefore $m_t < \tau_0$ for all $t$. We can take a
feasible resolution for each $R_t$ and $Q_2$ to obtain the
required feasible resolution for $P \in X$.

\noindent
{\bf Case 2.} $f_{\tau_0}(z,u)=(h_{\tau_0/2}(z,u))^2$ is a perfect square.\\
We need to make a coordinate change so that $P \in X$ is rewritten
as $$(x^2+2x h_{\tau_0/2}(z,u) + y^2 +
f_{\tau_0+1}(z,u)+f_{>\tau_0+1}(z,u)=0) \subset \bC^4.$$
 Depending on the parity
of $\tau_0$,  we consider $\wBl_v: Y_1 \to X$
 with weights
$v=\frac{1}{2}(\frac{\tau_0}{2}+2, \frac{\tau_0}{2}+1,1,1)$ or
$\frac{1}{2}(\frac{\tau_0}{2}+1, \frac{\tau_0}{2}+2,1,1)$. Without
loss of generality, we study the first case.
 Now $E =(\bfy^2+ 2 \bfx h_{\tau_0/2}+f_{\tau_0+1}(\bfz,\bfu)=0)\subset
\bP(\frac{\tau_0}{2}+2, \frac{\tau_0}{2}+1, 1,1)$.

 Easy computation yields the following:
\begin{enumerate}
\item $Y \cap U_2$ is non-singular and $Y_1 \cap U_1$ is singular only at $Q_1$, which is a terminal quotient
singularity of index $\frac{\tau_0}{2}+2$.

\item $Y \cap U_3$ is defined by
$$(\tilde{\varphi}: x^2z+2x h_{\tau_0/2}(1,u) + y^2 +
\tilde{f}_{\ge \tau_0+1}=0) \subset \bC^4.$$

For any singularity $R \in \Sing(Y)_{\ind=1} \cap U_3$, we write $R=(0,0,\alpha,\beta)$ and consider that coordinate change that $\bar{x}:=y, \bar{y}:=x, \bar{z}:=z-\alpha, \bar{u}:=u-\beta$.  Then $R$ is at worst a
 $cD$-like point and $\mu^\flat (R) \le  {\tau_0} -2$.

\item The same holds  for singularity in  $\Sing(Y)_{\ind=1} \cap U_4$.
\end{enumerate}

As a summary, one sees that $\Sing(Y)$ consist of a terminal quotient
singularity $Q_1$ and possibly some  $cD$-like points $R_t$ with
$\mu^\flat(R_t)< \tau_0$. We can take a feasible resolution for each $R_t$
and $Q_1$ to obtain the required feasible resolution for $P \in X$.
\end{proof}

\begin{prop} Given a  $cD$-like point with $\mu^\flat(P \in X)= \mu_0$.
Suppose that feasible resolutions exist for $cAx/2$-like point with $
\tau \le \mu_0$ and feasible resolutions exist for $cD$-like point
with $\mu^\flat < \mu_0$. Then there is a feasible resolution  for $P
\in X$.
\end{prop}

\begin{proof}
We always fix a normal form once and for all. By Lemma \ref{mu3},
we may assume that $\mu_0 \ge 4$. We set $\mu':=\lfloor
\frac{\mu_0}{2} \rfloor$. We consider divisorial contraction $Y
\to X$ with weights $(\mu',\mu'-1,2,1)$. Recall that $P \in X$ is
given by
$$ (\varphi: x^2+y^2z+\lambda yu^l + \sum a_{ij}z^iu^j=0).$$
 We may write $f_{2\mu'}:=\sum_{2i+j=2\mu'} a_{ij} z^i u^j= \prod_t
(\alpha_t z+\beta_tu^2 )^{m_t}$.

\noindent
{\bf Case 1.} $\lambda=0$. \\
It is straightforward to see that only singularity on $U_1 \cup U_2$ is $Q_2$, which is a terminal quotient singularity.
On $U_3 \cup U_4$, for any singularity $R \in \Sing(Y)_{\ind=1}\subset \Sing(E)$, then $R$ correspond to a factor $(\alpha_t z+\beta_t
u^2)^{m_t}$ with $m_t \ge 2$.
We distinguishes the following three subcases.

\noindent
{\bf  1-1.} $R \ne Q_3,Q_4$.\\
By changing coordinates $\bar{z}:=z+\frac{\beta_t}{\alpha_t}$, one
sees that $R$ is a $cA$ point. Hence feasible resolution over $R$
exists.

\noindent
{\bf 1-2.} $R=Q_3$.\\
Since $Y \cap U_3$ is defined by
$$ (\tilde{\varphi}: x^2+y^2+ \sum a_{ij} z^{2i+j-2 \mu'}u^j=0) \subset \bC^4/\frac{1}{2}(1,0,1,1) .$$
One sees that $Q_3$ is a $cAx/2$-like point with
$$\tau(Q_4) = \min \{2i+j-2\mu'+j| a_{ij} \ne 0\} \le j_0 \le \mu(P)=\mu^\flat(P)=\mu_0.$$
Feasible resolution over $Q_3$ exists by our hypothesis.

\noindent
{\bf 1-3.} $R=Q_4$. \\
Since $Y \cap U_4$ is defined by
$$(\tilde{\varphi}:  x^2+y^2z + \sum a_{ij} z^i u^{2i+j-2\mu'}=0) \subset \bC^4.$$
Then $Q_4$ is  a $cD$-like point with $\lambda=0$ as well.
 Since $\mu(P)=2i_0+j_0$ for some $z^{i_0}u^{j_0} \in \varphi$, one sees that
$$ \mu(Q_4) = \min \{ 2i+j-2\mu'+2i| a_{ij} \ne 0\}  \le
 \mu(P)-2\mu' +2i_0\le \mu(P), \eqno(\dagger)$$
 where the last inequality follows from
 $ \mu'=\lfloor \frac{\mu_0}{2}\rfloor \ge \lfloor \frac{2i_0}{2}\rfloor$.

 One can easily check  that $$\begin{array}{l}\mu^\flat(Q_4)=\mu(Q_4)  \le
 \mu(P)= \mu^\flat(P); \\
  \mu^\sharp(Q_4) \le
 \mu^\sharp(P)+2-2\mu' < \mu^\sharp(P). \end{array}$$

 By inductively on $\mu^\sharp$, there exist feasible resolution for $P \in X$.







\noindent
{\bf Case 2.} $\lambda \ne 0$.\\
\noindent
{\bf Subcase 2-1.} $2l-2 = \mu(P)$.\\
We proceed as in Case 1 and see that $\Sing(Y)=\{Q_2, Q_3\}$, where $Q_2$ is a terminal quotient singularity and $Q_3 \in Y\cap U_3$ is given by
$$ (\tilde{\varphi}: x^2+y^2-\frac{1}{4}\lambda^2 u^{2l}+ \sum a_{ij} z^{2i+j-2 \mu'}u^j=0) \subset \bC^4/\frac{1}{2}(1,0,1,1),$$
after completing the square.
One sees that $Q_3$ is a $cAx/2$-like point with
$$\tau(Q_4) \le \min \{2i+j-2\mu'+j| a_{ij} \ne 0\}  \le \mu(P)=\mu^\flat(P)=\mu_0.$$
Feasible resolution over $Q_3$ exists by our hypothesis.

\noindent
{\bf Subcase 2-2.}  $2l-2 > \mu(P)$.\\
We can proceed as in Case 1 all the way to equation $\dagger$.
 Therefore, $l(Q_4)=l(P)-\mu'-1$ and  $$ \mu^\flat(Q_4)= \min \{ 2l-2\mu'-4,
\mu(Q_4)\} \le \min\{ 2l-2, \mu(P)\} = \mu^\flat(P).$$

Inductively,  we are reduced to
either $\mu^\flat < \mu_0$ or $2l-2 =\mu$. Hence feasible resolution exists.


\noindent
{\bf Subcase 2-3.} $2l-2 < \mu(P)$.\\
For any $z^i u^j \in f$, one has $2i+j > 2l-2$ and hence $i+j \ge l$. We consider $\wBl_v: Y \to X$ with $v=(l,l,1,1)$ instead.

By Lemma \ref{nonsing}, $Y\cap U_4$ is nonsingular and hence $Y \to X$ is a divisorial contraction by Theorem \ref{terminal}. One sees that
$\Sing(Y)=\{Q_2, Q_3\}$, where $Q_2$ is a terminal quotient singularity and $Q_3 \in Y\cap U_3$ is given by
$$ (\tilde{\varphi}: x^2+y^2 +\lambda yu^{l}+ \sum a_{ij} z^{i+j-2l }u^j=0) \subset \bC^4,$$
which is a $cD$-like point. For a $cD$-like point, we introduce
$$ \rho^\sharp(P \in X) := \min\{i+j| z^iu^j \in \varphi, j=0 \text{ or } 1\},$$ which is finite and $\frac{1}{2} \mu \le \rho^\sharp \le \mu$.  Compare $Q_4$ with $P$, we have
$ \rho^\sharp(Q_4)  \le  \rho^\sharp(P)+1-2l$. Repeat the process $t$-times, we have $Q_{t,4} \in Y_t \to \ldots \to Y_1=Y \to X \ni P$ such that for $t$ sufficiently large
$$ \mu(Q_{t,4}) \le 2 (\mu (P)+t (1-2l)) < \mu_0(P).$$
 Hence  we are reduced to the situation  $\mu^\flat(Q_{t,4}) \le \mu(Q_{t,4}) < \mu_0$, or Subcase 2-1, or Subcase 2-2 in finite steps and  feasible resolution over $P \in X$ exists by our hypothesis.

\end{proof}

Combining all the above results in this section, we have the
following:
\begin{thm} \label{cD}
There is a feasible resolution  for any singularity of type $cD$ or $cAx/2$.
\end{thm}
\section{resolution of $cAx/4, cD/2, cD/3$  points}
In \cite{HaII}, Hayakawa shows that there is a partial resolution
$$ X_n \to ... \to X_1 \to X \ni P,$$
for a point $P \in X$ of index $r >1$ such that $X_n$ has only
terminal singularities of index $1$ and each map is a divisorial
contraction with minimal discrepancies.  If $\Sing(X_n)_{\ind=1}$ is
either of type $cA$ or $cD$, then feasible resolution exists by
the result of previous sections.

In fact, the partial resolution was constructed by picking any
divisorial contraction with minimal discrepancy at each step.
Therefore, for our purpose, it suffices to pick one divisorial
contraction $Y \to X$ over a given higher index point $P \in X$ of
type $cAx/4$, $cD/2$, $cD/3$, or $cE/2$ and verify that
$\Sing(Y)_{\ind=1}$ is either of type $cA$ or $cD$.

\begin{lem} \label{cAx4}
Given $P \in X$ of type $cAx/4$, there is a divisorial contraction
$Y \to X$ with discrepancy $\frac{1}{4}$ such that $\Sing(Y)_{\ind=1}$ is of type $cA$ or $cD$.
\end{lem}

\begin{proof}
We may write $P \in X$ as
$$ (\varphi: x^2+y^2+f(z,u)=x^2+y^2+\sum_{i+j=2l+1 \ge 3} a_{ij} z^{2i}
u^j=0) \subset \bC^4/\frac{1}{4}(1,3,1,2).$$

Let $\sigma(P \in X):= \min \{ i+j| a_{ij} \ne 0\}$, then we may
write $f(z,u)=f_\sigma(z,u)+f_{> \sigma}(z,u)$.

\noindent
{\bf Case 1.} $f_\sigma(z,u)$ is not a perfect square.\\
Depending on parity of $\frac{\sigma-1}{2}$, we consider $\wBl_v: Y \to X$ with weights $v=\frac{1}{4}(\sigma+2, \sigma, 1,2)$
or $\frac{1}{4}(\sigma, \sigma+2, 1,2)$. By \cite[Theorem 7.4]{HaI},
this is the only divisorial contraction.

Without loss of generality, we study the first weight. By Lemma
\ref{nonsing}, \ref{cA}, we have $\Sing(Y)_{\ind=1} \cap U_i$ is empty
for $i=1,2$. Moreover,  $Q_4$ is a $cD/2$-like point of index $2$. It
suffices to consider $U_3$.

In $U_3$,  $Y \cap U_3$ is defined by
$$ (\tilde{\varphi}: x^2z+y^2+\tilde{f}(z,u)=0) \subset \bC^4.$$ Therefore, it is immediate to see that $\Sing(Y) \cap U_3$ is either of type $cA$ or $cD$.

\noindent
{\bf Case 2.} $f_\sigma(z,u)=-h(z,u)^2$ is  a perfect square.\\
Depending on parity of  $\frac{\sigma-1}{2}$, we need to make a
coordinate change so that $P \in X$ is written as
$$(\varphi: x^2+2x h(z,u)+y^2+ f_{> \sigma}(z,u)=0) \subset \bC^4/\frac{1}{4}(1,3,1,2),$$
or
$$(\varphi: y^2+2y h(z,u)+x^2+ f_{> \sigma}(z,u)=0)\subset \bC^4/\frac{1}{4}(1,3,1,2).$$

 We consider weighted blowup $Y \to
X$ with weights $\frac{1}{4}(\sigma+4, \sigma+2, 1,2)$ or
$\frac{1}{4}(\sigma+2, \sigma+4, 1,2)$ respectively. By
\cite[Theorem 7.4]{HaI}, this is a divisorial contraction.

Without loss of generality, we study the first weight. By Lemma
 \ref{nonsing}, we have $\Sing(Y)_{\ind=1} \cap U_i$ is empty
for $i=1,2$. Then $R \ne Q_4$ for $Q_4$ is a $cD/2$-like point of
index $2$. It suffices to consider $U_3$.
Indeed,  $Y \cap U_3$ is defined by
$$ (\tilde{\varphi}: x^2z+2xh_\sigma(1,u)+y^2+\tilde{f}(z,u)=0) \subset \bC^4.$$ Therefore, it is immediate to see that $\Sing(Y) \cap U_3$ is either of type $cA$ or $cD$.
\end{proof}

\begin{lem} \label{cD2}
Given $P \in X$ of type $cD/2$, there is a divisorial contraction $Y
\to X$ with discrepancy $\frac{1}{2}$ such that $\Sing(Y)_{\ind=1}$ is of type $cA$ or $cD$.
\end{lem}

\begin{proof} By Mori's classification \cite{Mo82, YPG},
one has that $P \in X$ is given by $(\varphi=0) \subset
\bC^4/\frac{1}{2}(1,1,0,1)$ with $\varphi$ being one the following

$$ \left\{ \begin{array}{ll}
x^2+yzu+y^{2a}+u^{2b}+z^c, & a \ge b \ge 2, c \ge 3 \\
x^2+y^2 z+\lambda y u^{2l+1} + f(z,u^2).
\end{array}
\right.$$

\noindent {\bf Case 1.} $ \varphi=x^2+yzu+y^{2a}+u^{2b}+z^c$.\\
We take weighted blowup $Y \to X$ with weights
$v=\frac{1}{2}(3,1,2,1)$ (resp. $\frac{1}{2}(3,1,2,3)$) if $a=b=2$
(resp. $a \ge 3$), which is a divisorial contraction by
\cite{HaII}. Note that $wt_v(yzu)=wt_v(\varphi)$. Hence $uz$
(resp. $yu, yz$) appears in the equation of $Y \cap U_2$ (resp.
$U_3, U_4$). By Corollary \ref{cAq}, we conclude that
$\Sing(Y)_{\ind=1} \cap U_i$ is of type $cA$ for $i=2,3,4$.
Together with Lemma \ref{u1}, then we are done with this case.

\noindent {\bf Case 2.} $\varphi=x^2+y^2 z+\lambda y u^{2l+1} + f(z,u^2)$.\\
We may write $f(z,u^2)=\sum a_{ij}z^i u^{2j} \in
(z^3,z^2u^2,u^4)\bC\{z,u^2\}$. We define
$$ \left\{ \begin{array}{l}
\sigma:= \min \{ 2i+2j| z^iu^{2j} \in f\};\\
\sigma^\flat:= \min \{ 2l-1, \sigma\}.
\end{array}
\right. $$ Note that we have $l \ge 1$ and $\sigma \ge 2$ for $P \in
X$ is a $cD/2$ point.

\noindent
{\bf Subcase 2-1.} $l=1$.\\
We consider weighted blowup $Y \to X$ with weight
$v=\frac{1}{2}(2,1,2,1)$, which is a divisorial contraction by \cite{HaII}. Now $$E =( \bfx^2+\bfy^2 \bfz+ \lambda
\bfy \bfu^3+f_{wt_v=2}=0) \subset \bP(2,1,2,1).$$ By Lemma \ref{u1},
\ref{nonsing}, one sees that $\Sing(Y)_{\ind=1} \cap U_i$ is empty for
$i=1,2,4$. Since $Q_3$ is at worst of type $cAx/2$. We are done.

\noindent
{\bf Subcase 2-2.} $l \ge 2$ and $\sigma=2$.\\
One has $f_{\sigma=2}=u^4$, in particular, $u^4 \in f$. We consider
weighted blowup $Y \to X$ with weight $v=\frac{1}{2}(2,1,2,1)$
again. Now $$E =( \bfx^2+\bfy^2 \bfz+ \bfu^4=0) \subset
\bP(2,1,2,1).$$ We thus have $\Sing(Y)_{\ind=1} \subset \Sing(E)
\subset \{Q_3\}$. However, $Q_3$ is a point of index $2$. We are
done.

\noindent
{\bf Subcase 2-3.} $ \sigma^\flat \ge 3$.\\
Let $$ \sigma':= 2 \lfloor \frac{\sigma^\flat -1}{2} \rfloor
+1=\left\{
\begin{array}{ll}
\sigma^\flat &\text{ if $\sigma^\flat$ is odd;}\\
 \sigma^\flat-1
&\text{ if $\sigma^\flat$ is even;}
\end{array} \right.$$

We consider weighted blowup $Y \to X$ with weight
$v=\frac{1}{2}(\sigma',\sigma'-2,4,1)$, which is a divisorial contraction by \cite{HaII}. Clearly, by Lemma \ref{u1},
\ref{nonsing}, one sees that $\Sing(Y)_{\ind=1} \cap U_i$ is empty for
$i=1,2$. Since $Q_3$ is a point of index $4$, it remains to consider
$U_4$. Now $Y \cap U_4$ is given by
$$ (\tilde{\varphi}: x^2+y^2z+\lambda yu^{({2l-1-\sigma'})/{2}}+\tilde{f}=0) \subset \bC^4.
$$

It follows that $Y \cap U_4$ is at worst of type $cD$.
 We are done.
\end{proof}

\begin{lem} \label{cD3}
Given $P \in X$ of type $cD/3$, there is a divisorial contraction $Y
\to X$ with discrepancy $\frac{1}{3}$ such that $\Sing(Y)_{\ind=1}$ is of type $cA$ or $cD$.
\end{lem}

\begin{proof} By Mori's classification \cite{Mo82, YPG}, one has that $P \in X$ is given as $(\varphi=0) \subset
\bC^4/\frac{1}{3}(0,2,1,1)$ with $\varphi$ being one of the
following:
$$ \left\{ \begin{array}{ll}
x^2+y^3+zu(z+u); & \\

x^2+y^3+zu^2+ yg(z,u)+h(z,u); &g \in \frak{m}^4, h \in \frak{m}^6;
\\

x^2+y^3+z^3+yg(z,u)+h(z,u); &g \in \frak{m}^4, h \in \frak{m}^6.

\end{array} \right.
$$

\noindent {\bf Case 1.} $\varphi$ is one of the first two
cases.\\
By \cite[Theorem 9.9, 9.14, 9.20]{HaI}, the weighted blowup $Y \to
X$ with weight $\frac{1}{3}(3,2,4,1)$ is a divisorial contraction.
Now $$E = \left\{ \begin{array}{l} \bfx^2+\bfy^3+\bfz\bfu^2=0 \text{ or }\\
\bfx^2+\bfy^3+\bfz\bfu^2 + \lambda \bfy \bfu^4 + \lambda' \bfu^6=0
\end{array} \right\} \subset \bP(3,2,4,1),$$
for some $\lambda, \lambda'$ respectively. It is easy to check that
$\Sing(E)=Q_3$ and hence $\Sing(Y)_{\ind=1}$ is empty for the first two
cases.

\noindent{\bf Case 2.} $ \varphi=x^2+y^3+z^3+yg(z,u)+h(z,u)$.\\
\noindent {\bf Subcase 2-1.} Either $u^4 \in g$ or $u^6 \in h$.\\
Then we consider the weighted blowup with weight
$\frac{1}{3}(3,2,4,1)$ again, which is a divisorial contraction (cf. \cite[Theorem 9.20]{HaI}. Now
$$E=( \bfx^2+\bfy^3+ \lambda \bfy \bfu^4 + \lambda' \bfu^6=0)
\subset \bP(3,2,4,1),$$ for some $(\lambda, \lambda') \ne (0,0)$.
One sees that $\Sing(E)=Q_3$ and hence $\Sing(Y)_{\ind=1}$ is
empty.

\noindent {\bf Subcase 2-2.}   $u^4 \not \in g$, $u^6 \not \in h$
and either
$zu^5 \in h$ or $u^9 \in h$.\\
 Then we consider the weighted blowup with weight
$\frac{1}{3}(3,2,4,1)$ which is a divisorial contraction. Now
$$E=( \bfx^2+\bfy^3=0) \subset \bP(3,2,4,1).$$ One sees that
$\Sing(E) \subset U_3 \cup U_4$. However, the equation of $Y \cap
U_4$ contains the term $zu$ or $u$ and hence contains at worst $cA$
points by Lemma \ref{cA}. Together with the fact that $Q_3$ is a
$cAx/4$ point, we are done with this case.

\noindent {\bf Subcase 2-3.} $u^4 \not \in g$, all $zu^5, u^6, u^9
\not \in h$.\\
 Then we consider the weighted blowup $Y \to X$ with
weight $\frac{1}{3}(6,5,4,1)$, which is a divisorial contraction
by \cite[Theorem 9.25]{HaI}.

By Lemma \ref{nonsing}, $Y \cap U_2$ is
nonsingular away from $Q_2$, which is a quotient singularity of
index $5$. Together with $\Sing(Y) \cap U_1=\emptyset$ and $Q_3 \not \in Y$, it remains to check $Y \cap U_4$, which is defined by
$$(\tilde{\varphi}:x^2+y^3u+z^3+\text{others}=0) \subset \bC^4. $$
 which is at worst of type
$cE_6$. In fact, this corresponds to Case 1 and 2 of the proof of Theorem \ref{cE6}. Notice that in the proof, we use weighted blowups $\wBl_v$  with $v=(2,2,1,1),(3,2,1,1)$ or $(3,2,2,1)$. After weighted blowup, there could have singularities of type $cA, cD, cA/2$, $cAx/2$, and terminal quotients. We thus concludes that feasible resolution exists for this case.
\end{proof}

We thus conclude the section by the following:
\begin{thm} \label{hind}
There is a feasible resolution for any singularity of type $cAx/4$,
$cD/3$, or $cD/2$.
\end{thm}

\section{resolution of $cE$ and $cE/2$ points}
Recall that a $cE$ point has the following description.
$$(\varphi: x^2+y^3+f(y,z,u)= x^2+y^3+y g(z,u)+h(z,u)=0) \subset \bC^4.$$
An isolated singularity with the above desription is called a $cE$-like singularity.

For a polynomial ( resp. formal power series) $G(z,u) \in \bC[z,u]$
(resp. $ \bC[[z,u]]$), we define

$$\tau(G):=\min\{j+k| z^ju^k \in G\}.$$

For $cE$ singularity, one has $\tau(g) \ge 3$ and $\tau(h) \ge 4$.
Moreover, either $\tau(g)=3$ or $\tau(h) \le 5$. More precisely,

\begin{enumerate}
\item It is $cE_6$ if $\tau(h)=4$ and $\tau (g) \ge 3$.

\item It is $cE_7$ if $\tau(h) \ge 5$ and $\tau(g)=3$.

\item It is $cE_8$ if $\tau(h)=5 $ and $\tau (g) \ge 4$.
\end{enumerate}

\begin{rem}
An isolated $cE$-like singularity is at worst of type $cD$ (resp.
$cE_6$, $cE_7$, $cE_8$) if $\tau(g) \le 2$ or $\tau (h) \le 3$
(resp. $\tau(h) \le 4$, $\tau(g) \le 3$, $\tau(h) \le 5$).
\end{rem}



\begin{setup} Notations and Conventions \end{setup}
{\bf 1.}  We fix the notation that  $g_3(z,u):=g_{\tau=3}(z,u),
h_4(z,u):=h_{\tau=4}(z,u)$ and $h_5(z,u):=h_{\tau=5}(z,u)$. In the
case of $cE_6$, $\tau(h)=4$. By replacing $z,u$ and up to a
constant, we may  and do assume that
$$ h_4 \in \{ z^4,  z^4+z^3u, z^4+2z^3u+ z^2u^2, z^4+z^2u^2, z^4+zu^3 \}.$$
In particular, $z^4 \in h_4$.

In the case of $cE_7$, $\tau(g)=3$.  We may and do assume that
$$g_3 \in \{ z^3,z^3+z^2u, z^3+zu^2\}.$$ In particular, $z^3 \in g_3$.

In the case of $cE_8$,  $\tau(h)=5$. We may  and do assume
that
$$ h_5 \in \{  z^5,z^5+z^4u,z^5+2z^4u+z^3u^2,z^5+z^3u^2,z^5+2z^4u-z^3u^2-2z^2u^3, z^5+z^2u^3,
z^5+zu^4 \}.$$ In particular, $z^5 \in h_5$.

{\bf 2.}  We define
$$\tau^*(\varphi):= \min\{p| y^iz^ju^p \in \varphi \text{ with }
i+j \le 1\}.$$

Since $P \in X$ is isolated, there is a term $yu^p, zu^p$ or
$u^p$ in $\varphi$ otherwise $P$ is singular along a line
$(x=y=z=0)$. Hence $\tau^*(\varphi)$ is a well-defined  integer.

{\bf 3.} For a weight $v=(a,b,k,1)$, we denote it $v_l$ with
$l=a+b+k-1$. In our discussion, we always consider weight $v_l$
such that $v_l(\varphi)=l$.

{\bf 4.} Fix a weight $v=\frac{1}{r}(a,b,k,1)$ with $r=1,2$, we
write
$$\varphi=x^2+y^3+y g_v+yg_{v+1}+y g_{>}+h_{v}+v_{v+1}+h_{>}, $$
where $g_v$ (resp. $g_{v+1}$) is the homogeneous part of $g(z,u)$
such that $wt_{v}(yg_v)=wt_v (\varphi)$ (resp.
$wt_{v}(yg_v)=wt_v (\varphi)+1$) and $g_{>}$ is the remaining
part with greater weight, and $h_v$ (resp. $h_{v+1}$) is the
homogeneous part of $h(z,u)$ with   $v$-weight  equal to
$wt_{v}(\varphi)$ (resp. $wt_{v}(\varphi)+1$) and $h_{>}$ is the
remaining part with greater weight.

{\bf 5.}  For simplicity of notation, sometime we may denote by
$g_{m}$ or $h_{m}$ for the $v$-homogeneous part with $v$-weight
equal to $m$. This notation should not be confused
with $g_3$ nor $g_v$.

\subsection{Some preparation}
The general strategy is as following. For a given $cE$ or $cE/2$
singularity $P \in X$. We consider weighted blowup $Y \to X$ with
weight $v=\frac{1}{r}(a,b,k,1)$ and $r=1,2$ such that
$\frac{1}{2}(a+b+k+1)-wt_v(\varphi)=1+\frac{1}{r}$. This is a weighted blowup
with discrepancy $\frac{1}{r}$ if $E$ is irreducible. We check
that $\Sing(Y) \cap U_4$ is isolated and each $R \in \Sing(Y) \cap
U_4$ is terminal. Then the weighted blowup $Y \to X$ is a
divisorial contraction with discrepancy $\frac{1}{r}$ by Theorem
\ref{terminal}.

Moreover, we check that each singular point $R \in \Sing(Y)_{\ind=1}$ is "milder" than $P \in X$ in the sense that either it is of milder type, or it can only admit smaller weight. We can prove the existence of feasible resolution by induction on types and weights.






\begin{setup} \label{u4}
We work on $Y \cap U_4$.

Now $Y \cap U_4$ is defined  by $\tilde{\varphi}$, which can be written as
$$\begin{array}{ll} \tilde{\varphi} &= x^2u^{wt_v(x^2)-wt_v(\varphi)}+ y^3u^{wt_v(y^3)-wt_v(\varphi)}\\
&+ y g_v(z,1)+yu g_{v+1}(z,1)+ y\widetilde{g_{>}} + h_{v}(z,1)+
uh_{v+1}(z,1)+ \widetilde{h_{>}}, \end{array} $$ such
that $u^2| \widetilde{g_{>}}$ and $u^2 | \widetilde{h_{>v}}$.




\begin{lem} \label{isolated}
Suppose that $wt_v(x^2)= wt_v(\varphi)$ or $wt_v(\varphi)+1$ and
$wt_v(y^3)=wt_v(\varphi)$. Then $\Sing(Y) \cap U_4$ is isolated
UNLESS:
$$\text{There is $s(z,u)$ such that} \left\{ \begin{array}{l}
g_v=-3 s(z,u)^2,\\
h_v=2 s(z,u)^3,\\
h_{v+1}=-s(z,u) g_{v+1}.
\end{array} \right. \eqno{\natural}
$$
\end{lem}

\begin{proof}
 If $2 wt_v(x)= wt_v(\varphi)$, we have $\tilde{\varphi}_x=2x$.
 If $2 wt_v(x)= wt_v(\varphi)+1$, then $Y \cap U_1$ is non-singular away from $Q_1$ by Lemma \ref{nonsing}.
 Hence we have $\Sing(Y) \cap U_4 \subset (x=0)$ in both cases.
 Moreover, $\Sing(Y) \subset E$, hence we have $\Sing(Y) \cap U_4 \subset (u=0)$.

 Therefore, we have
$$\begin{array}{ll} \Sing(Y) \cap U_4 &\subset (x=u=0) \cap (\tilde{\varphi}=\tilde{\varphi}_y =\tilde{\varphi}_u=0) \\ & \subset (x=u=0) \cap \Sigma, \end{array}$$
where $\Sigma$ is defined as $$\left\{ \begin{array}{l}
y^3+yg_v+h_v=0,\\
3y^2+g_v=0,\\
yg_{v+1}+h_{v+1}=0.
\end{array} \right.
$$

If $g_v$ is not a perfect square, then $3y^2+g_v$ is irreducible
and hence $\Sigma$ is finite. If $g_v$ is a perfect square, then
we write it as $g_v=-3 s^2$. One sees that $\Sigma$ is finite
unless $y-s$ or $y+s$ divides the above three polynomials. The
statement now follows.
\end{proof}

\begin{lem} \label{isolated1} Suppose more generally that
$$ \varphi=x^2+y^3+s(z,u) y^2 + y g_v+ y
g_{v+1}+yg_>+h_v+h_{v+1}+h_{>}.$$ Suppose that $wt_v(x^2)=
wt_v(\varphi)$  and $wt_v(y^3)=wt_v(\varphi)+1$. Then $\Sing(Y)
\cap U_4$ is isolated UNLESS  $g_v=h_v=h_{v+1}=0$.
\end{lem}

\begin{proof}
 Since $wt_v(y^3)= wt_v(\varphi)+1$, then $Y \cap U_2$ is non-singular away from $Q_2$.
 Hence we have $\Sing(Y) \cap U_4 \subset (x=y=0)$.
 Moreover, $\Sing(Y) \subset E$, hence we have $\Sing(Y) \cap U_4 \subset (u=0)$.

Therefore, we have
$$\begin{array}{ll} Siny(Y) \cap U_4 &\subset (x=y=u=0) \cap (\tilde{\varphi}=\tilde{\varphi}_y =\tilde{\varphi}_u=0) \\ & \subset (x=y=u=0) \cap ( h_v(z,1)=g_v(z,1)=\tilde{\varphi}_u=0) \\ & \subset (x=y=u=0) \cap ( h_v(z,1)=g_v(z,1)=h_{v+1}(z,1)=0). \end{array}$$
The statement now follows.
\end{proof}





\end{setup}

\begin{setup} \label{u4}
We study the most common case that
$$  wt_v(x^2)=wt_v(y^3)=wt_v(\varphi). $$\end{setup}

Suppose furthermore that $\natural$ does not hold, then $\Sing(Y)
\cap U_4$ is isolated. We now study the possible type of these
singularities.

Notice that we have at least one of $g_v,h_v, h_{v+1}$ is
non-zero, otherwise $\natural$ holds.

\noindent
{\bf Case 1.} $h_v \ne 0$.\\
We write $$h_v(z,u)=\lambda  z^{m'} u^{n'} \prod (z-\alpha_t u^k
)^{l'_t}.$$ Then $$k \cdot \tau(h) \ge wt_v(h_\tau) \ge
wt_v(\varphi) =n'+k(m'+\sum l'_t ). \eqno{\dag_h}$$ In particular,
$$ \tau(h) \ge m'+\sum l'_t.$$

  Then $E \cap U_4$ is
defined by $$\begin{array}{l} (x^2+y^3+ y g_v(z,1)+h_v(z,1)=0)
\\
=(x^2+y^3+ y g_v(z,1)+\lambda' z^{m'} \prod (z-\alpha'_t
)^{l'_t}=0) \subset U_4 \cong \bC^4,
\end{array}$$ which is irreducible.


It is easy to see that
\begin{itemize}
\item if $m'+\sum l'_t \le 2$ then $E \cap U_4$ is  at worst Du
Val of $A$-type and  hence $\Sing(Y) \cap U_4$ is at worst of $cA$
type;

\item if $m'+\sum l'_t = 3$ then $E \cap U_4$ is  at worst Du Val
of $D$-type and  hence $\Sing(Y) \cap U_4$ is at worst of $cD$
type;

\item if $m'+\sum l'_t = 4$ then $E \cap U_4$ is  at worst Du Val
of $E_6$-type and  hence $\Sing(Y) \cap U_4$ is at worst of $cE_6$
type;

\item if $m'+\sum l'_t = 5$ then $E \cap U_4$ is  at worst Du Val
of $E_8$-type and  hence $\Sing(Y) \cap U_4$ is at worst of $cE_8$
type;
\end{itemize}

Notice also that if $P \in X $ is of type $cE_6$ (resp. $cE_8$),
then $\Sing(Y) \cap U_4$ is at worst of $cE_6$ (resp. $cE_8$).

In any event, $Y \cap U_4$ is terminal. By Theorem \ref{terminal},
$Y \to X$ is a divisorial contraction with discrepancy $1$.

\noindent
{\bf Case 2.} $g_v \ne 0$.\\
We write $$g_v(z,u)=\lambda z^m u^n \prod (z-\alpha_t
u^k)^{l_t}.$$ Then $$k \cdot \tau(g) \ge wt_v(g_\tau) \ge
wt_v(\varphi)-b =n+k(m+\sum l_t ). \eqno{\dag_g}$$ In particular,
$$ \tau(g) \ge m+\sum l_t.$$

Then $E \cap U_4$ is defined by $$\begin{array}{l} (x^2+y^3+ y
g_v(z,1)+h_v(z,1)=0)
\\
=(x^2+y^3+ \lambda y z^m \prod (z-\alpha_t )^{l_t}+h_v(z,1)=0)
\subset U_4 \cong \bC^4,
\end{array}$$ which is irreducible.

It is easy to see that
\begin{itemize}
\item if $m+\sum l_t \le 1$ then $E \cap U_4$ is  at worst Du Val
of $A$-type and  hence $\Sing(Y) \cap U_4$ is at worst of $cA$
type;

\item if $m+\sum l_t = 2$ then $E \cap U_4$ is  at worst Du Val of
$D$-type and  hence $\Sing(Y) \cap U_4$ is at worst of $cD$ type;

\item if $m+\sum l_t = 3$ then $E \cap U_4$ is  at worst Du Val of
$E_7$-type and  hence $\Sing(Y) \cap U_4$ is at worst of $cE_7$
type.

\end{itemize}

Notice also that if $P \in X $ is of type $cE_7$, then $\Sing(Y)
\cap U_4$ is at worst of $cE_7$.

\noindent
{\bf Case 3.} $h_v =0, h_{v+1} \ne 0$.\\
We write
$$h_{v+1}(z,u)=\lambda''  z^{m''} u^{n''} \prod (z-\alpha_t u^k )^{l''_t}.$$
Then $$k \cdot \tau(h) \ge wt_v(h_\tau) \ge wt_v(\varphi)
=n''+k(m''+\sum l''_t )-1. \eqno{\dag'_h}$$ In particular, if
$k>1$, then we still have
$$ \tau(h) \ge m''+\sum l''_t.$$
The same conclusion as in Case 1 still holds.

In any event, $Y \cap U_4$ is terminal. By Theorem \ref{terminal},
$Y \to X$ is a divisorial contraction.

As a summary, we conclude that
\begin{thm} \label{divcontr}
Given $P\in X$ a $cE$ point defined by $(\varphi: x^2+y^3+y
g(z,u)+h(z,u)=0)$. Let $Y \to X$ be a weight blowup with weight
$v=(a,b,k,1)$ that $k>1$. Suppose that $wt_v ( x^2)=wt_v
(y^3)=wt_v(\varphi)$ and $\natural$ does not hold.  Then $Y \to X$
is a divisorial contraction. Also, any  singularity on $Y \cap
U_4$ is at worst of type $cE_6$ (resp. $cE_7, cE_8$) if $P \in X$
is of type $cE_6$ (resp. $cE_7, cE_8$).
\end{thm}

\begin{rem} \label{Q4} Consider the case that $P \in X$ is of type $cE_6$.
Suppose the worst case that $Y$ has a singularity $R$ of type
$cE_6$. This happens only when $h_v=z^{4}$ or $h_v= (z-\alpha_t u
)^{4}$ for $m'+\sum l'_t \le 4$. If $h_v=(z-\alpha'_tu^k)^{4}$. By
considering  the weight-invariant coordinate change that
$\bar{z}=z-\alpha'_tu^k$, we may and do assume that $R=Q_4$ and
$Q_4$ is the unique singularity in $U_4$.

We can make the same assumption if $P \in X$ is of type $cE_7$,
$cE_8$.
\end{rem}


\begin{prop} \label{x0y0} Let  $Y \to X$ be a weighted blowup of a
$cE $ point with weight $v=(a,b,k,1)$. Suppose  that
$wt_v(x^2)=wt_v(y^3)=wt_v(\varphi)$. If any one of $g_v=0$,
$h_v=0$,  $n
>0$,  or $n' >0$ holds,
 then $\Sing(E) \cap U_2-U_4=\emptyset$.

In particular, if $\natural$ does not hold and $v=v_{30}, v_{24},
v_{18}, v_{12}$, then $Y \to X$ is a divisorial contraction and
$\Sing(Y)_{\ind=1}  \subset U_4$.
\end{prop}

\begin{proof}  In affine
coordinate $U_2$, $E$ is defined by
$$\tilde{\phi}: x^2+1+ g_v(z,u)+h_{v}(z,u)=0,$$ with
$g_v(z,u)$ (resp. $h_v(z,u)$) being homogeneous with respect to
the weight $v$ of weight $wt_v(\varphi)-b$ (resp.
$wt_v(\varphi)$). It follows that $$\left\{ \begin{array}{l} kz
\frac{\partial g_v(z,u)}{\partial z}+u \frac{\partial
g_v(z,u)}{\partial u} = (wt_v(\varphi)-b)\cdot g_v(u,v), \\
kz \frac{\partial h_v(z,u)}{\partial z}+u \frac{\partial
h_v(z,u)}{\partial u} = wt_v(\varphi)\cdot h_v(u,v)
\end{array} \right.$$

 It follows that  $$\psi_1:=wt_v(\varphi)
\tilde{\phi}-a x\tilde{\phi}_{x}-k
z\tilde{\phi}_{z}-u\tilde{\phi}_{u}=wt_v(\varphi)+bg_v(u,v),$$
where $g_v(u,v)=\lambda z^m u^n \prod (z-\alpha_t u^k )^{l_t} $.
Note that $\psi_1$ must be satisfied at any singular point
$\Sing(E) \cap U_2$.
 If $\lambda=0$, then one sees that $\Sing(E) \cap U_2=
 \emptyset$.

If $n >0$, then $\Sing(E) \cap U_2 \not \subset U_4$
 otherwise $u=0$ will leads to a contradiction.

If we consider $$\psi_2:=(wt_v(\varphi)-b) \tilde{\phi}-(a-b/2)
x\tilde{\phi}_{x}-k z\tilde{\phi}_{z}-u\tilde{\phi}_{u}$$
$$=(wt_v(\varphi)-b)-bh_v(z,u),$$ where $h_v(u,v)= \lambda' z'^m u'^n \prod (z-\alpha'_t u^k
)^{l'_t}$. Then one sees similarly that $\Sing(E) \cap
U_2=\emptyset$ if $\lambda'=0$ and $ \Sing(E) \cap U_2 \not
\subset U_4$ if  $n'>0$.

We now prove the second statement. Since $\natural$ does not hold,
hence $Y \to X $ is a divisorial contraction. Therefore,
$\Sing(Y)_{\ind=1} \subset \Sing(E)$. We have that $\Sing(E) \cap
U_1=\emptyset$.  Notice that either $g_v=0$ or $n>0$ for $v_{12},
v_{24}, v_{30}$. Also one has either $h_v=0$ or $n'>0$ for
$v_{18}$. Therefore, $\Sing(Y)_{\ind=1} \cap (U_1 \cup U_2 \cup U_4)
\subset U_4$. Finally $Q_3$ is of index $>1$. This completes the
proof.
\end{proof}



\subsection{Resolution of $cE_6$ points}

In this subsection, we shall prove that
\begin{thm} \label{cE6}
There is a feasible resolution for any $cE_6$ singularity.
\end{thm}

\begin{proof}
We will need to consider weighted blowup $ Y \to X$ with the
following weights $v_{12}=(6,4,3,1)$, $v_8=(4,3,2,1)$,
$v_6=(3,2,2,1)$, $v_4=(2,2,1,1)$ and $v_5=(3,2,1,1)$. It is
sufficient to show that $\Sing(Y)_{\ind>1}$ is not of type $cE/2$
and there exists feasible resolution on $\Sing(Y)_{\ind=1}$.

\noindent
{\bf Case 1.} $wt_{v_{12}}(f) <12$, $wt_{v_{8}}(f) <8$ and $wt_{v_{6}}( f) < 6$. \\
\noindent {\bf Subcase 1-1.} $h_4$ is not a perfect square.

 We
consider the weighted blowup $Y \to X$ with weight
$v_{4}=(2,2,1,1)$. It is clear that $E$ is irreducible if $h_4$ is
not a perfect square.

Since  $wt_{v_{6}}( f) < 6$, we must have a term $\theta \in f$
such that  $wt_{v_{6}}( \theta) < 6$. One has that $\theta =yu^3$
or $\theta= u^5$.

\noindent
{\bf Claim.} $Y \to X$ is a divisorial contraction. \\
To see this, if $\theta=u^5$, then it follows that $Y \cap U_4$ is nonsingular and
thus $ Y \to X$ is a divisorial contraction by Theorem
\ref{terminal}. If $\theta=yu^3$, then $\natural$ does not hold and hence $Y \cap U_4$ has at worst singularities of type $cA$. Therefore,  $ Y \to X$ is a divisorial contraction by Theorem
\ref{terminal}.

Clearly, $Y \cap U_1$ is nonsingular. Moreover, $Y \cap U_2$ has
singularity of type $cAx/2$ at $Q_2$ and at worst of type $cA$ for
points other than $Q_2$. Since $z^4 \in h_4$, we have $Q_3 \not
\in Y$. Therefore, feasible resolution exists for this case.



\noindent {\bf Subcase 1-2.} $h_4$ is a perfect square, i.e.
$h_4=z^4$ or $z^4+2z^3u+z^2u^2$.

Since  $wt_{v_{6}}( f) < 6$, we  have either $yu^3$ or $u^5$ in
$f$.
 Write $h_4=-q(z,u)^2$. Consider the coordinate change
$\bar{x}:=x-q(z,u)$, we have
$$\bar{\varphi}:= \bar{x}^2+2 \bar{x} q(z,u)+y^3 + y g(z,u)+h_{\tau \ge 5}(z,u).$$
We consider weighted blowup with weight $v_5=(3,2,1,1)$ instead.
Note that we still have either $yu^3$ or $u^5 \in \bar{\varphi}$.

 By Lemma \ref{nonsing}, $Y \cap U_i$ is nonsingular away
from $Q_i$ for $i=1,2,4$ and $Q_1, Q_2$ are terminal quotient
singularity of index $3,2$ respectively.
 By Theorem \ref{terminal}, $Y \to X $ is a divisorial
contraction. It remains to consider $Q_3$. Since $z^4 \in h_4$, we
have $\bar{x}z^2 \in \bar{\varphi}$.  Hence $Y \cap U_3$ is also
non-singular by Lemma \ref{nonsing}. We thus conclude that
$\Sing(Y)=\Sing(Y)_{\ind >1}$ consists of $Q_1,Q_2$,
 which are terminal quotient singularities of index $3,2$ respectively.

\noindent
{\bf Case 2} $wt_{v_{12}}(f) <12$, $wt_{v_{8}}(f) <8$ and $wt_{v_{6}}( f) \ge 6$. \\
We consider the weighted blowup $Y \to X$ with weight
$v_{6}=(3,2,2,1)$. It is clear that $E$ is irreducible.
There is a term $\theta \in f$ with  $wt_{v_{8}}(\theta) < 8$ and
$wt_{v_{6}}(\theta) \ge 6$. One sees that $$\theta \in \{yzu^2,
yu^4, z^3u, z^2u^2, z^2u^3, zu^4, zu^5, u^6,u^7\}. $$
It follows in particular that at least one of $g_v, h_v, h_{v+1}$ is non-zero.

\noindent
{\bf Claim.} $Y \to X$ is a divisorial contraction.\\
To see this, suppose first that $\natural$ holds, then $g_v=-3s(z,u)^2$ for some $s(z,u) \ne 0$.
We may assume that $s(z,u)=u^2$ and hence $yu^4 \in \varphi$. Then $Y \cap U_4$ is nonsingular by
Lemma \ref{nonsing} and hence $Y \to X$ is a divisorial
contraction by Theorem \ref{terminal}.

Suppose that $\natural$ does not hold. Then $\Sing(Y) \cap U_4$ is isolated.
In $U_4$, the
corresponding term $\tilde{\varphi}$  of $\theta$ in $\tilde{\varphi}$ is
$$\tilde{\theta} \in \{ yz, y, z^3, z^2, z^2u, z, zu,  1,u\}.$$
Hence $\Sing(Y) \cap U_4$ is at worst of
type $cD$. By Theorem \ref{terminal}, $Y \to X$ is a divisorial
contraction. This proved the Claim.

 We consider $Y \cap U_3$. We have that $z^4 \in h_4$ and
hence $Q_3$ is at worst  of type $cA/2$. By Corollary
\ref{cAq}, $\Sing(Y)_{\ind=1} \cap U_3$ is at worst of $cA$ type.
By Lemma \ref{u1}, $\Sing(Y)_{\ind=1} \cap U_1 = \emptyset$.
Together with $Q_2 \not \in Y$, we concludes that
$\Sing(Y)_{\ind=1}$ is at worst of type $cD$ and
$\Sing(Y)_{\ind>1}=\{Q_3\}$, of type $cA/2$. Feasible resolution exists for this
case.

\noindent
{\bf Case 3} $wt_{v_{12}}(f) <12$ and $wt_{v_{8}}( f) \ge 8$. \\
We consider the weighted blowup $Y \to X$ with weight
$v_{8}=(4,3,2,1)$.

{\bf 1.} Note that $\tau(h)=4$ and $wt_{v_8}(h) \ge 8$, we thus
have $h_4=z^4 \in h_v \ne 0$. By Lemma \ref{isolated1}, $\Sing(Y)
\cap U_4$ is isolated. Also, one has $Q_3 \not \in Y$.

{\bf 2.} Since $wt_{v_{12}}(f) < 12$ and $wt_{v_{8}}(f) \ge 8$,
there is a term $\theta=y^iz^ju^k \in f$ with $wt_{v_{12}}(\theta)
< 12$ and $wt_{v_{8}}(\theta) \ge 8$.
%
Hence the corresponding term $\tilde{\theta}=y^iz^ju^{k'} \in
\tilde{\varphi}$ satisfying $$i+j+k'=i+j+(3i+2j+k-8) \le 3.$$
One can verify that $\Sing(Y) \cap U_4$ is at worst of type $cE_6$ with $h_4$ has at least two factors. 
Hence if there is a $cE_6$ points then it  is  in Case 1 or 2.
By Theorem \ref{terminal}, $Y \to X$ is a divisorial contraction.

{\bf 3.} By Lemma \ref{nonsing}, $\Sing(Y) \cap U_2 = \{Q_2 \}$ and
$Q_2$ is a terminal quotient singularity of index $3$. Also
$\Sing(Y)_{\ind=1} \cap U_1 = \emptyset$ by Lemma \ref{u1}.

{\bf 4.} We summarize that $\Sing(Y)_{\ind>1}$ consists of $Q_2$,
which is a quotient singularity of index 3 and  two terminal
quotient singularities of index $2$ in the line $(\bfy=\bfu=0)
\subset E$ and $\Sing(Y)_{\ind=1} \subset U_4$ are at worst of
type $cE_6$ in Case 1 or 2.

\noindent
{\bf Case 4.} $wt_{v_{12}}( f) \ge 12$.\\
We consider $Y \to X$ the weighted blowup with weight
$v_{12}=(6,4,3,1)$.

{\bf 1.} Since $wt_{v_{12}}( h) \ge 12$ and $\tau(h)=4$, we have
$z^4 \in h_{v}$. It follows that  $Q_3 \not \in Y$,  $h_v \ne
2s^3$, and thus $\natural$ does not holds.  By  Theorem
\ref{divcontr} and Proposition \ref{x0y0}, the weighted blowup $Y
\to X$ is a divisorial contraction with discrepancy $1$. Moreover,
$\Sing(Y)_{\ind=1}  \subset U_4$ are at worst of type $cD$ or
there is only a unique point $R \in Y$ of type of type $cE_6$.


{\bf 2.} It follows that feasible resolution exists for $P \in X$
unless that $  \Sing(Y)_{\ind=1} = R \in Y$ is of type $cE_6$. In
fact, if it is of type $cE_6$, we may assume that $R=Q_4$ (cf.
Remark \ref{Q4}).

Clearly, $$\left\{ \begin{array}{l} \tau^*(\tilde{\varphi}) <
\tau^*(\varphi), \\
wt_{v_{12}}(\tilde{\varphi}) \le wt_{v_{12}}(\varphi). \end{array}
\right.$$

The existence of feasible resolution is thus reduced to  $cE_6$
singularities with $wt_{v_{12}} < 12$ by induction on $\tau^*$.


{\bf 3.} We remark that $\Sing(Y)_{\ind>1}$ consists of  terminal quotient
singularities on the line $(\bfz=\bfu=0)$ and $(\bfy=\bfu=0)$ of index $2,3$
respectively.


This  exhausts all cases of type $cE_6$.
We thus conclude that for a given $P \in X$ of type $cE_6$, there is a feasible partial resolution $Y_s \to \ldots \to Y_1=Y \to X$ such that
$\Sing(Y_s)_{\ind=1}$ are at worst $cD$ and $\Sing(Y_s)_{\ind >1}$ can only be of type $cA/2, cA/2$ or terminal quotient. Hence  feasible resolution exists for $Y_s$ and hence for $P \in X$.
\end{proof}

\subsection{Resolution of $cE/2$ points}
It is convenient to consider $cE/2$ points before we move into the
$cE_7$ and $cE_8$ singularities. Given a $cE/2$ point $P \in X$,
which is given by
$$( \varphi =x^2+y^3+ \sum a_{ij}yz^ju^k + \sum b_{jk} z^ju^k=0) \subset \bC^4/ \frac{1}{2}(1,0,1,1),$$
 with $h_4:=\sum_{j+k=4} b_{jk}z^ju^k \ne 0$.

We will consider weighted blowup with weights
$v_1=\frac{1}{2}(3,2,3,1)$ or $v_2=\frac{1}{2}(5,4,3,1)$. Note
that  $wt_{v_1} (z) =wt_{v_2}(z), wt_{v_1} (u) =wt_{v_2}(u)$.
Hence may simply denote it as  $wt_{3,1}(G)$ for $G \in
\bC[[z,u]]$.

\begin{thm} \label{cE/2}
There is a feasible resolution for any $cE/2$ singularity.
\end{thm}

\begin{proof}
We first consider weighted blowup $ Y \to X$ with weight
$v=\frac{1}{2}(3,2,3,1)$.
 As before, we can rewrite $\varphi$  as
$$\varphi =x^2+y^3+ y g_v+yg_{v+1}+yg_>+h_v+h_{v+1}+h_>.$$

Notice that  Lemma \ref{isolated} still holds in the current situation.

\noindent
{\bf Case 1.} $wt_{3,1}(h_4) = 3$, i.e. $u^4 \not \in h_4, zu^3 \in h_4$.\\
%
It is straightforward to see  that $E$ is irreducible and $Y \cap U_4$ is nonsingular,
hence $Y \to X$ is a divisorial contraction. Also $Y \cap U_3$ has
singularity $Q_3$ of type $cD/3$, might have terminal quotient
singularity of index $3$ along the line $\bfy=\bfu=0$ and might have
singularity at worst of type $cD$. There is no other singularity.

\noindent
{\bf Case 2.} $wt_{3,1}(h_4) = 4$, i.e. $u^4, zu^3 \not \in h_4, z^2u^2 \in h_4$.\\
Since $z^2u^2 \in h_{v+1}$, one sees that $\natural$ does not hold
and hence $Y \cap U_4$ has only isolated singularities.

It is straightforward to see  that $Y \cap U_4$ might have
singularities at worst of type $cD$, hence $Y \to X$ is a
divisorial contraction.  On $Y \cap U_3$, there are a  singularity
$Q_3$ of type $cD/3$, possibly terminal quotient singularities of
index $3$ along the line $(\bfy=\bfu=0)$ and possibly
singularities at worst of type $cD$. There is no other singularity
outside $U_3 \cup U_4$.

\noindent
 {\bf Case 3.} $wt_{3,1}(h_4)  \ge 5$  and  $\natural$ does not hold. \\
 The similar argument works.
Indeed, there is a term $\theta \in \varphi$
 among $\{yu^4, yzu^3, yu^6, zu^5,u^6, u^8\}$. The corresponding
 term in $\tilde{\varphi}$ the equation of $Y \cap U_4$ is among
 $\{y, yzu, yu, zu, 1, u\}$. It is easy to see that singularities
 are at worst of type $cD$ or $cD/3$ as in Case 2.

\noindent
 {\bf Case 4.} $wt_{3,1}(h_4)  \ge 5$ and $\natural$  holds. \\
We then consider  a coordinate change that $\bar{y}:=y-\lambda
u^2$ for some $\lambda$  so that we may rewrite $P \in X$ as
$$ \bar{\varphi}= x^2+\bar{y}^3+3s \bar{y}^2+\bar{y} g_{v+1}+ \bar{y} g_{v+2} +\bar{y} {g}_> + \bar{h}_{v+2}+\bar{h}_{v+3}+\bar{h}_> $$ similarly.

Since $wt_{3,1}(h_4)  \ge 5$, one has  either $z^3u$ or $z^4 \in
\varphi$. It follows that either $z^3u$ or $z^4 \in
\bar{\varphi}$.

\noindent {\bf Subcase 4-1.} Suppose that there is a term
$\theta=y^iz^ju^k \in \bar{\varphi}$ such that $6i+5j+k \le 16$.
We consider weighted blowup $ Y \to X$ with weight
$\frac{1}{2}(5,4,3,1)$ instead. By Lemma \ref{isolated1}, $Y \cap
U_4$ is isolated. The corresponding term
$\tilde{\theta}=y^iz^ju^{k'}$ in $Y \cap U_4$ satisfying $$
i+j+k'=j+(3j+k-10)/2 \le 3.$$ One sees that $Y \cap U_4$ has at
worst $cE_6$ singularities. Hence $Y \to X$ is a divisorial
contraction.

Moreover, $\Sing(Y) \cap U_i=\{Q_i\}$ for $i=2,3$, which is a
terminal quotient singularity of index $4$ and $3$. Also $Y \cap
U_1$ is non-singular. Therefore feasible resolutions exist.

\noindent {\bf Subcase 4-2.} Suppose that there is no term
$\theta=y^iz^ju^k \in \bar{\varphi}$ such that $6i+5j+k \le 16$.
We consider weighted blowup $Y \to X$ with weight
$v_3=\frac{1}{2}(9,6,5,1)$ instead. Note that in this situation,
$wt_{v_3} \bar{\varphi}=9$ and $z^4 \in \bar{\varphi}$. It is easy
to see that $\Sing(Y) \cap U_4 $ is isolated by Lemma
\ref{isolated} or by direct computation. Indeed, $Y \cap U_4$ has
at worst singularities of type $cE_6$. Hence $Y \to X$ is a
divisorial contraction.

Moreover, $\Sing(Y)_{\ind>1}=\{Q_3\}$ which is of index $5$.
Another higher index point is a point $R \in (\bfz=\bfu=0)$, which
is terminal quotient of index $3$.  We thus conclude that a
feasible resolution exists for any $cE/2$ point by Theorem
\ref{cE6} and results in previous sections.
\end{proof}

\subsection{Resolution of $cE_7$ points}
In this subsection, we consider $cE_7$ points.
\begin{thm} \label{cE7}
There is a feasible resolution for  any $cE_7$ singularity.
\end{thm}

\begin{proof}  We shall consider weights $v_{18}=(9,6,4,1)$,
$v_{14}=(7,5,3,1)$, $v_{12}=(6,4,3,1)$, $v_9=(5,3,2,1)$,
$v_8=(4,3,2,1)$, $v_6=(3,2,2,1)$, $v_5=(3,2,1,1)$ and discuss as
in $cE_6$ case.

 \noindent
{\bf Case 1.} $wt_{v_{18}}( f) < 18$,$\ldots,wt_{v_{6}}( f) < 6$.\\
We consider weighted blowup with weight $v_{5}=(3,2,1,1)$.

 Since  $z^3 \in g_v$, we have that $yz^3 \in f$, $E$ is irreducible,  and $Y \cap U_3$ is non-singular. Hence, $Y \to X$ is a
divisorial contraction by Theorem \ref{terminal}.

By Lemma  \ref{nonsing},  $Y \cap U_i$ is non-singular
away from $Q_i$ for $i=1,2,3$.
 Notice that there is a term $\theta$
with  $wt_{v_{6}}(\theta) <6$ and $wt_{v_{5}}( \theta) \ge 5$. It
follows that $\theta$ is $yu^3$ or $u^5$. Hence $Q_4$ is either
non-singular or $Q_4 \not \in Y$. Therefore,
$\Sing(Y)=\Sing(Y)_{\ind
>1}=\{Q_1,Q_2\}$, which are terminal quotient points of index $3$ and $2$ respectively.

\noindent
{\bf Case 2.}$wt_{v_{18}}( f) < 18$,$\ldots,wt_{v_{8}}( f) < 8$, $wt_{v_{6}}( f) \ge 6$.\\
We consider weighted blowup with weight $v_{6}=(3,2,2,1)$ and proceed as in Case 2 of $cE_6$, then $\wBl_{v_6}: Y \to X$ is
a divisorial contraction and $\Sing(Y) \cap U_4$ is at worst of type $cD$.

 We consider $Y \cap U_3$. Since $yz^3 \in \varphi$, one sees that $Q_3$ is at worst  of type $cD/2$. By Corollary
\ref{cDq}, $\Sing(Y)_{\ind=1} \cap U_3$ is at worst of type $cD$.
By Lemma \ref{u1}, $\Sing(Y)_{\ind=1} \cap U_1 = \emptyset$.
Together with $Q_2 \not \in Y$, we conclude that
$\Sing(Y)_{\ind=1}$ is at worst of type $cD$ and
$\Sing(Y)_{\ind>1}=\{Q_3\}$, of type $cD/2$. Feasible resolution exists for this
case.

 \noindent
{\bf Case 3.} $wt_{v_{18}}( f) < 18$,$\ldots,wt_{v_{9}}( f) < 9$, $wt_{v_{8}}( f) \ge 8$.\\
We consider weighted blowup with weight $v_{8}=(4,3,2,1)$.

There is a term $\theta=y^{i}z^{j}u^{k}$ satisfying $
wt_{v_{9}}(\theta) <9$  and $ wt_{v_{8}}(\theta) \ge 8 $. Hence
either $g_v$ or  $h_v$ contains $\theta$ and is non-zero. By Lemma \ref{isolated1}, $\Sing(Y)
\cap U_4 $ is isolated.

By the same argument as in Case 3 of $cE_6$, 
one sees that $\Sing(Y) \cap U_4$ is
at worst of type $cE_6$. This implies in particular that $Y \to X$
is a divisorial contraction.

Since both $y^3, yz^3$ are in $\varphi$, by Lemma \ref{nonsing},
one has $Y \cap U_i$ is nonsingular away from $Q_i$ for $i=2,3$.
Together with $\Sing(Y)_{\ind=1} \cap U_1=\emptyset$, we are done.

\noindent
{\bf Case 4.} $wt_{v_{18}}( f) < 18$,$\ldots,wt_{v_{12}}( f) < 12$, $wt_{v_{9}}( f) \ge 9$.\\
We consider weighted blowup with weight $v_{9}=(5,3,2,1)$. One has
$z^3  \in g_v \ne -3s^2$. Hence $\natural$ does not hold and
$\Sing(Y)\cap U_4$ is isolated by Lemma \ref{isolated}.


We consider $Y \cap U_4$. Since $ wt_{v_{12}}(\theta) <12$  and $
wt_{v_{9}}(\theta) \ge 9 $ for some $\theta=y^{i}z^{j}u^{k} \in
\varphi$, we have $\tilde{\theta}=y^{i}z^{j}u^{k'} \in
\tilde{\varphi}$ with $i+j+k'=4i+3j+k-9 \le 2$. It follows easily
that $Y \cap U_4$ has at worst singularity of type $cA$ by Lemma
\ref{cA}. Therefore, $Y \to X$ is a divisorial contraction with
discrepancy $1$.

By Lemma \ref{nonsing}, one sees that $Y \cap U_i$ is non-singular
away from $Q_i$ for $i=1,3$. Moreover, $Q_2 \not \in Y$.  Feasible resolution exists for this case.

 \noindent
{\bf Case 5.}$wt_{v_{18}}( f) < 18$, $wt_{v_{14}}( f) < 14$, $wt_{v_{12}}( f) \ge 12$.\\
We consider the weighted blowup with weight $v_{12}=(6,4,3,1)$.
One has $z^3 \in g_{v+1} \ne 0$.


\noindent {\bf Subcase 5-1.} Suppose that $\natural$ does not hold. \\
Then $Y \to X$ is a divisorial contraction and $\Sing(Y)_{\ind=1}
\subset U_4$ by Proposition \ref{x0y0}.

Indeed, by the discussion in \ref{u4}, we may assume that
$\Sing(Y)_{\ind=1}$ singularities at worst of type $cD$ or a
singularity of type $cE_7$ at $Q_4$.
Clearly, we have $$ \left\{ \begin{array}{l}
\tau^*(\tilde{\varphi}) < \tau^*({\varphi}), \\
wt_{v_{18}}(\tilde{\varphi}) \le
wt_{v_{18}}({\varphi}), \\
wt_{v_{14}}(\tilde{\varphi}) \le wt_{v_{14}}({\varphi}).
\end{array} \right.$$
By induction on $\tau^*$, we are reduced to the case that $wt_{v_{12}} < 12$.

\noindent {\bf Subcase 5-2.} Suppose that $\natural$ hold. \\
Notice that there is $\theta \in \varphi$ with
$\varphi_{v_{14}}(\theta) <14,  \varphi_{v_{12}}(\theta) \ge 12$, it
is easy to see that  $\theta \in y g_v$ , $h_v$ or in $h_{v+1}$.
This implies in particular that $s \ne 0$.
 We  consider a coordinate change that
$\bar{y}:=y-s(z,u)$ for some $s=\alpha zu+ \beta u^4$    so that
we may write $P \in X$ as
$$ \bar{\varphi}= x^2+\bar{y}^3+3s \bar{y}^2+\bar{y} g_{v+1}+ \bar{y} \bar{g}_> + \bar{h}_>$$
similarly.

We consider weighted blowup with weight $v_{14}=(7,5,3,1)$ instead
in this situation.
Since  $\bar{y}z^3 \in \bar{\varphi}$, by Lemma
\ref{isolated1}, one sees that $\Sing(Y) \cap U_4$ is isolated.
 One can check that $\Sing(Y) \cap U_4$ has
at worst singularity of type $cD$ for $s \ne 0$. Therefore, $Y \to X$ is a divisorial contraction.

One can easily check that for $i=2,3$, $\Sing(Y) \cap U_i= \{Q_i\}$, which is terminal quotient of index $5$ and $3$ respectively.  Moreover,
$Q_1 \not \in Y$ and hence there exists a feasible resolution.

\noindent
{\bf Case 6.} $wt_{v_{18}}( f) < 18$, $wt_{v_{14}}( f) \ge 14$.\\
We consider the weighted blowup with weight $v_{14}=(7,5,3,1)$.
Similarly, one has $g_{3}=z^3$ and $z^3  \in g_v  \ne 0$. By Lemma
\ref{isolated1}, $\Sing(Y) \cap U_4$ is isolated.
Since $ wt_{v_{18}}(\theta) <18$  and $ wt_{v_{14}}(\theta) \ge 14
$ for some $\theta=y^{i}z^{j}u^{k} \in \varphi$, we have
$\tilde{\theta}=y^{i}z^{j}u^{k'} \in \tilde{\varphi}$ with
$i+j+k'=6i+4j+k-14 \le 3$. One can verify that  any $R \in \Sing(Y)
\cap U_4$  is at worst of type $cE_6$. Therefore $Y \to X $ is a
divisorial contraction.

By Lemma \ref{nonsing}, we have that $Y \cap U_i$ is nonsingular
away from $Q_i$ for $i=2,3$.
 Moreover $Q_1 \not \in Y$, hence feasible resolution exists for this case.

\noindent
{\bf Case 7.} $wt_{v_{18}}( f) \ge 18$\\
We consider the weighted blowup with weight $v_{18}=(9,6,4,1)$.
Since $wt_{v_{18}}( g) \ge 18$ and $\tau(g)=3$, we have
$g_{3}=z^3$ and $z^3  \in g_v  \ne -3s^2$. It is clear that
$\natural$ does not holds.
By Proposition \ref{x0y0}, $Y \to X$ is a divisorial contraction
and $\Sing(Y)_{\ind=1} \subset U_4$.


 $\Sing(Y) \cap U_4$ is at worst of type $cE_7$. If there
is a singularity of type $cE_7$, then we proceed by induction in $\tau^*$.
Then it can be reduced to
the  cases with $wt_{v_{18}} < 18$.


This completes the proof that a feasible resolution exist for $cE_7$ singularity.
\end{proof}

\subsection{Resolution of $cE_8$ points}



In this subsection, we shall prove that
\begin{thm} \label{cE8}
There is a feasible resolution for  any $cE_8$ singularity.
\end{thm}

\begin{proof}
We will need to consider  weights $v_{30}=(15, 10, 6, 1)$, $v_{24}=(12, 8, 5, 1)$,...etc.

 \noindent
{\bf Case 1.}$wt_{v_{30}}(f)  < 30,\ldots,wt_{v_{8}}(f) <
8$.\\
We consider weighted blowup with weight $v_6=(3,2,2,1)$. By 6.2,
we have that $wt_{v_{6}}(f) \ge 6$ always holds and $z^5 \in h$.
We consider weighted blowup with weight $v_{6}=(3,2,2,1)$ and proceed as in Case 2 of $cE_6$, then $\wBl_{v_6}: Y \to X$ is
a divisorial contraction and $\Sing(Y) \cap U_4$ is at worst of type $cD$.

 We consider $Y \cap U_3$. We have that $z^5 \in \varphi$.
Hence $Q_3$ is at worst  of type $cE/2$ and  $\Sing(Y)_{\ind=1} \cap U_3$ is at worst of type $cE_6$.
By Lemma \ref{u1}, $\Sing(Y)_{\ind=1} \cap U_1 = \emptyset$.
Together with $Q_2 \not \in Y$, we conclude that
feasible resolution exists for this
case.

\noindent {\bf Case 2.}$wt_{v_{30}}(f)  < 30,\ldots,wt_{v_{9}}(f)
<
9$,  $wt_{v_{8}}(f) \ge 8$\\
We consider weighted blowup with weight $v_8=(4,3,2,1)$. By the
same argument as in Case 3 of $cE_7$, one has that $Y \to X$ is a
divisorial contraction with $\Sing(Y) \cap U_4$  at worst of type
$cE_6$.

Since  $y^3$ is in $\varphi$, by Lemma \ref{nonsing}, one has that $Y
\cap U_2$ is nonsingular away from $Q_2$. Together with
$\Sing(Y)_{\ind=1} \cap U_1=\emptyset$ and $Q_3$ is of type $cAx/2$, we are done.

 \noindent
{\bf Case 3.} $wt_{v_{30}}(f)  < 30,\ldots,wt_{v_{12}}(f) <
12$,  $wt_{v_{9}}(f) \ge 9$\\
We consider weighted blowup with weight $v_9=(5,3,2,1)$. Since
$h_5 \ne 0$, we have either $z^4u$ or $z^5 \in h$. Now the same
argument as in Case 4 of $cE_7$ goes through.

 \noindent
{\bf Case 4.}$wt_{v_{30}}(f)  < 30,\ldots,wt_{v_{14}}(f) < 14$,
$wt_{v_{12}}(f) \ge 12$.

 We consider weighted blowup with weight
$v_{12}=(6,4,3,1)$. The proof is essentially parallel to Case 5 of $cE_7$.
Note that we have $h_5=z^5$ or $z^4u \in f$.

\noindent {\bf Subcase 4-1.} Suppose $\natural$ does not hold.

Then $Y \to X$ is a divisorial contraction and $\Sing(Y)_{\ind=1}
\subset U_4$ by Theorem \ref{divcontr} and Proposition \ref{x0y0}.

Indeed by the discussion in \ref{u4}, we know that either
$\Sing(Y)_{\ind=1}$ consists of singularities at worst of type $cE_6$ or we may
assume that $Q_4$, the only  singularity  in $U_4$, is  of type
$cE_8$.

Clearly, we have $$ \left\{ \begin{array}{l}
\tau^*(\tilde{\varphi}) < \tau^*({\varphi}), \\
wt_{v_{l}}(\tilde{\varphi}) \le
wt_{v_{l}}({\varphi}), \\
\end{array} \right.$$ for all $l \ge 12$.
By induction on $\tau^*$, we are done.

\noindent
{\bf Subcase 4-2.} Suppose that $\natural$ hold. \\
As in Subcase 5-2 of $cE_7$, we consider a coordinate change and then the weighted blowup with weight $v_{14}=(7,5,3,1)$ instead
in this situation.
Since  $z^5 \in \bar{\varphi}$, by Lemma
\ref{isolated1}, one sees that $\Sing(Y) \cap U_4$ is isolated.
 One can check that $\Sing(Y) \cap U_4$ has
at worst singularity of type $cD$ for $s \ne 0$. Therefore, $Y \to X$ is a divisorial contraction.

One can easily check that for $i=2,3$, $\Sing(Y) \cap U_i= \{Q_i\}$, which is terminal quotient of index $5$ and $3$ respectively.  Moreover,
$Q_1 \not \in Y$ and hence there exists a feasible resolution.

 \noindent
{\bf Case 5.}$wt_{v_{30}}(f)  < 30$, $\ldots, wt_{v_{18}}(f) <
18$,  $wt_{v_{14}}(f) \ge 14$\\
We can proceed as  in Subcase 6 of $cE_7$. Since $z^5 \in h_{v+1} \ne 0$, we still have that $\Sing(Y) \cap U_4$ has
isolated singularities by Lemma \ref{isolated1}. Thus the same conclusion holds.

 \noindent
{\bf Case 6.}$wt_{v_{30}}(f)  < 30$, $wt_{v_{24}}(f)  < 24, wt_{v_{20}}(f) <
20$,  $wt_{v_{18}}(f) \ge 18$\\
We consider weighted blowup with weight $v_{18}=(9,6,4,1)$.

Since there is a term $\theta \in f$ with $wt_{v_{20}}(\theta)
<20$ and $wt_{v_{18}}( \theta) \ge 18$. This implies
$\tilde{\theta}$ is in $g_v, h_v$ or $h_{v+1}$.

\noindent
{\bf Subcase 6-1.} Suppose $\natural$ does not hold. \\
Then $Y \to X$ is a divisorial contraction by Theorem \ref{divcontr}.

One sees that  $Y \cap U_3$ is given by
$(\tilde{\varphi}: x^2+y^3+z^2+\text{other terms}=0) \subset \bC^4/\frac{1}{4}(1,2,3,1).$
Therefore,  $\Sing(Y) \cap U_3$ is type $cAx/4$ or $cA$.
 We also have $\Sing(Y)_{\ind=1} \cap U_1=\emptyset$, and $Q_2
\not \in Y$.

It remains to consider $Y \cap U_4$.
Notice that the corresponding term $\tilde{\theta}=y^iz^ju^{k'}$ has $i+j+k' \le 4$ unless $\theta=z^4u^3, \tilde{\theta}=z^4u$.
If $i+j+k' \le 4$, then $Y \cap U_4$ has  singularities at worst of type $cE_7$ and hence feasible resolution exists.  Suppose that  $\tilde{\theta}=z^4u
\in \tilde{\varphi}$. Hence $$wt_{v_{l}}(\tilde{\varphi})
  <l,$$  for $l=30, 24, 20, 18$. Therefore, $Y \cap U_4$ has singularities at worst of type $cE_8$ in Subcase 1-4.

\noindent
{\bf Subcase 6-2.} Suppose that $\natural$ hold. \\
We first consider a coordinate change that $\bar{y}:=y-s(z,u)$ with $s(z,u) \ne 0$ since there is $\theta$ in $yg_v, h_v$ or $h_{v+1}$.
 Now $P \in X$ is defined as
$$ \bar{\varphi}= x^2+\bar{y}^3+3s \bar{y}^2+\bar{y} g_{v+1}+ \bar{y} \bar{g}_> + \bar{h}_>$$
and we consider weighted blowup with weight $v_{20}=(10,7,4,1)$ instead in this situation.

Since $z^5 \in \varphi$ and hence  $z^5 \in \bar{\varphi}$,  one sees that $\Sing(Y) \cap U_4$ is isolated,
by Lemma \ref{isolated1}.
Since $s=(\alpha zu^2+ \beta u^6) \ne
0$, we have either $\bar{y}^2 zu^2$ or $\bar{y}^2u^6 \in
{\bar{\varphi}} $. One can check that  $Y \cap U_4$ has at worst singularities of type $cD$ and thus
 $Y \to X $ is a divisorial contraction.

Together with the fact that $\Sing(Y)_{\ind=1} \cap U_1$ is empty,
$Y \cap U_2$ is non-singular away from $Q_2$ and  $Q_3 \in
\Sing(Y)_{\ind>1}$, one sees that  a feasible resolution exists.

 \noindent{\bf Case 7.} ${v_{30}}(f)  < 30$, $wt_{v_{24}}(f) <
24$,  $wt_{v_{20}}(f) \ge 20$\\
We consider the weighted blowup with weight $v_{20}=(10,7,4,1)$.
Since $z^5 \in h_v$, we have $Q_3 \not \in Y$ and $\Sing(Y) \cap
U_4$ is isolated by Lemma \ref{isolated1}.

 We work on $U_4$. There is a term $\theta=y^iz^ju^k \in f$ with
$wt_{v_{24}}(\theta) <24$ and $wt_{v_{20}}( \theta) \ge 20$. Hence
$\tilde{\theta}=y^i z^j u^{k'}$ with $i+j+k' \le 3$. It follows
that $\Sing(Y) \cap U_4$ is at worst of type $cE_6$.

Therefore, $Y \to X$ is a divisorial contraction. Together with
the fact that $\Sing(Y)_{\ind=1} \cap U_1$ is empty, $Y \cap U_2$ is
non-singular away from $Q_2$ and  $Q_3 \in \Sing(Y)_{\ind>1}$, one
sees that  a feasible resolution exists.

 \noindent
{\bf Case 8.}$wt_{v_{30}}(f)  < 30$, $wt_{v_{24}}(f) \ge 24$.\\
We consider the weighted blowup with weight $v_{24}=(12,8,5,1)$.
One notices that $\tau(h)=5$ implies that $z^5 \in h_{v+1}$.
Since $wt_{v_{24}}(g_{v+1})=17$ and hence $u^2|g_{v+1}$. It follows that  $z^5 \not \in s(z,u) g_{v+1}$ and $\natural$ does not hold.
Therefore $Y
\to X$ is a divisorial contraction and $\Sing(Y)_{\ind=1} \subset
U_4$ by Theorem \ref{divcontr} and Proposition \ref{x0y0}.

Unless $\Sing(Y) \cap U_4=\{Q_4\}$ is of type $cE_8$, we have
feasible resolution of $Y$.
If $\Sing(Y) \cap U_4=\{Q_4\}$ is of type $cE_8$, then we have
$$\left\{ \begin{array}{l} \tau^*(\tilde{\varphi}) < \tau^*(\varphi); \\
wt_{v_{30}}(\tilde{\varphi})  \le wt_{v_{30}}(\varphi) ;\\
wt_{v_{24}}(\tilde{\varphi})  \le wt_{v_{24}}(\varphi).
\end{array}  \right.$$

 By induction on $\tau^*$,
the existence of feasible resolution is thus reduced to
the existence of feasible resolution of milder singularity or to
the existence of feasible resolution  of $cE_8$ singularities with
$wt_{v_{24}} < 24$.

\bigskip \noindent
{\bf Case 9.} $wt_{v_{30}}(f) \ge 30$.\\
We consider the weighted blowup with weight $v_{30}=(15,10,6,1)$.
The similar argument as in Case 8 works.

This completes the
proof.
\end{proof}

\begin{proof}[Proof of  Main Theorem] This follows from
Theorem \ref{cArsln}, \ref{cAr}, \ref{cD}, \ref{hind}, \ref{cE6},
\ref{cE/2}, \ref{cE7}, \ref{cE8}.
\end{proof}


\end{document}